\documentclass{article}
\usepackage{amssymb}
\usepackage{amsmath}
\usepackage{amsfonts}
\usepackage{mathtools}
\usepackage{amsthm}

\newtheorem{theorem}{Theorem}

\newtheorem{corollary}[theorem]{Corollary}

\newtheorem{definition}[theorem]{Definition}

\newtheorem{lemma}[theorem]{Lemma}

\newtheorem{remark}[theorem]{Remark}

\renewcommand{\thefootnote}{\alph{footnote}}

\usepackage{xcolor}

\usepackage{authblk}
\title{The new notion of Bohl dichotomy for nonautonomous difference equations and its relation to exponential dichotomy}

\usepackage[hang,flushmargin]{footmisc}

\author[1]{Adam Czornik}
\author[2]{Konrad Kitzing}
\author[2]{Stefan Siegmund}
\affil[1]{Faculty of Automatic Control, Electronics and Computer Science, Silesian University of Technology, Gliwice, Poland}
\affil[2]{Institute of Analysis, Faculty of Mathematics, TU Dresden, Germany}
\date{\today}

\newcommand\blfootnote[1]{%
  \begingroup
  \renewcommand\thefootnote{}\footnote{#1}%
  \addtocounter{footnote}{-1}%
  \endgroup
}

\begin{document}

\maketitle

\blfootnote{
Adam Czornik: adam.czornik@polsl.pl\\
Konrad Kitzing: konrad.kitzing@tu-dresden.de\\
Stefan Siegmund: stefan.siegmund@tu-dresden.de
}

\begin{abstract}
In \cite{CzornikEtal2023} the concept of Bohl dichotomy is introduced which is a notion of hyperbolicity for linear nonautonomous difference equations that is weaker than the classical concept of exponential dichotomy. In the class of systems with bounded invertible coefficient matrices which have bounded inverses, we study the relation between the set $\mathrm{BD}$ of systems with Bohl dichotomy and the set $\mathrm{ED}$ of systems with exponential dichotomy. It can be easily seen from the definition of Bohl dichotomy that $\mathrm{ED} \subseteq \mathrm{BD}$. Using a counterexample we show that the closure of $\mathrm{ED}$ is not contained in $\mathrm{BD}$. The main result of this paper is the characterization $\operatorname{int}\mathrm{BD} = \mathrm{ED}$. The proof uses upper triangular normal forms of systems which are dynamically equivalent and utilizes a diagonal argument to choose subsequences of perturbations each of which is constructed with the Millionshikov Rotation Method. An Appendix describes the Millionshikov Rotation Method in the context of nonautonomous difference equations as a universal tool.
\end{abstract}



\newpage

The following is a draft of the article
\begin{center}
    The new notion of Bohl dichotomy for nonautonomous difference equations and its relation to exponential dichotomy
\end{center}
published in the \emph{Journal of Difference Equations and Applications}.

\section{Exponential and Bohl dichotomy}

Consider the system
\begin{equation}
x(n+1)=A(n)x(n),\qquad n\in \mathbb{N},  \label{1}
\end{equation}%
with invertible $A(n) \in \mathbb{R}^{d \times d}$ for each $n\in \mathbb{N}=\{0,1,\dots \}$. We denote the transition
matrix of system (\ref{1}) by $\Phi _{A}(n,m)$, $n,$ $m\in \mathbb{N}$, i.e.
\begin{equation*}
\Phi _{A}(n,m)=%
\begin{cases}
A(n-1)\cdots A(m) & \text{ for }n>m, \\
I & \text{ for }n=m, \\
\Phi _{A}(m,n)^{-1} & \text{ for }n<m,%
\end{cases}%
\end{equation*}%
where $I$ denotes the identity matrix in $\mathbb{R}^{d\times d}$. Any
solution $\left( x(n)\right) _{n\in \mathbb{N}}$ of \eqref{1} satisfies
\begin{equation*}
x(n)=\Phi _{A}(n,m)x(m),\qquad n,m\in \mathbb{N}.
\end{equation*}%
For every $k\in \mathbb{N}$ and $x_{k}\in \mathbb{R}^{d}$ the unique
solution of \eqref{1} which satisfies the initial condition $x(k)=x_{k}$ is
denoted by $(x(n,k,x_{k}))_{n\in \mathbb{N}}$ and for
short by $(x(n,x_{0}))_{n\in \mathbb{N}}$ if $k=0$. In particular,
\begin{equation*}
x(n,x_{0})=\Phi _{A}(n,0)x_{0},\qquad n\in \mathbb{N}.
\end{equation*}%
For $x \in \mathbb{R}^d$ and $M \in \mathbb{R}^{d\times d}$ we denote the Euclidean norm of $x$ by $\|x\|$ and the induced matrix norm of $M$ by $\|M\|$ (which is also called spectral norm of $M$).
Throughout the paper we assume that $A=(A(n))_{n\in \mathbb{N}}$ and $%
A^{-1} \coloneqq (A(n)^{-1})_{n\in \mathbb{N}}$ are bounded, i.e.\ $A\in \mathcal{L}^{%
\mathrm{Lya}}(\mathbb{N},\mathbb{R}^{d\times d}) \coloneqq \{B:B,B^{-1}\in \mathcal{L}%
^{\infty }(\mathbb{N},\mathbb{R}^{d\times d})\}$ is a so-called \emph{%
Lyapunov sequence}, where $\mathcal{L}^{\infty }(\mathbb{N},\mathbb{R}%
^{d\times d})$ denotes the Banach space of bounded sequences $B=(B(k))_{k\in
\mathbb{N}}$ in $\mathbb{R}^{d\times d}$ with norm $\Vert B\Vert _{\infty
}=\sup_{k\in \mathbb{N}}\Vert B(k)\Vert $.

For the definition of exponential dichotomy, a classical notion of hyperbolicity for system \eqref{1}, see e.g.\ \cite{CzornikEtal2023} and the references therein.
\begin{definition}[Exponential dichotomy]
\label{ED} System \eqref{1} has an \emph{exponential dichotomy (ED)} if
there exist subspaces $L_{1},L_{2}\subseteq \mathbb{R}^{d}$ with $\mathbb{R}%
^{d}=L_{1}\oplus L_{2}$, $\alpha >0$ and $K>0$ such that
\begin{align}
\Vert x(n,x_{0})\Vert & \leq K\mathrm e^{-\alpha (n-m)}\Vert x(m,x_{0})\Vert , &
\!\!x_{0}& \in L_{1},n\geq m,  \label{Dich1} \\
\Vert x(n,x_{0})\Vert & \geq K^{-1}\mathrm e^{\alpha (n-m)}\Vert x(m,x_{0})\Vert , &
\!\!x_{0}& \in L_{2},n\geq m.  \label{Dich2}
\end{align}
\end{definition}
In \cite[Definition 3]{CzornikEtal2023} the following weaker notion of hyperbolicity was introduced for system \eqref{1} to which we refer as Bohl dichotomy, see also \cite{Barreira2018}.
\begin{definition}[Bohl dichotomy]
\label{bohl} System \eqref{1} has a \emph{Bohl dichotomy (BD)} if there
exist subspaces $L_{1},L_{2}\subseteq \mathbb{R}^{d}$ with $\mathbb{R}%
^{d}=L_{1}\oplus L_{2}$, $\alpha >0$ and functions $C_{1},C_{2}\colon
\mathbb{R}^{d}\rightarrow \left( 0,\infty \right) $ such that
\begin{align}
\Vert x(n,x_{0})\Vert & \leq C_{1}(x_{0})\mathrm e^{-\alpha (n-m)}\Vert
x(m,x_{0})\Vert , & \!\!x_{0}& \in L_{1},n\geq m,  \label{11} \\
\Vert x(n,x_{0})\Vert & \geq C_{2}(x_{0})\mathrm e^{\alpha (n-m)}\Vert
x(m,x_{0})\Vert , & \!\!x_{0}& \in L_{2},n\geq m.  \label{12}
\end{align}
\end{definition}
The notion of Bohl dichotomy is thus weaker than that of exponential dichotomy, because the constants $C_1(x_0)$ and $C_2(x_0)$ in the estimates \eqref{11} and \eqref{12} do depend on the state variable $x_0$ in \(L_1\) and \(L_2\), respectively, whereas in \eqref{Dich1} and \eqref{Dich2} the estimates are uniform on \(L_1\) and \(L_2\).
With the abbreviations
\begin{equation*}
\mathrm{ED}^{d}\coloneqq\big\{A\in \mathcal{L}^{\mathrm{Lya}}(\mathbb{N},\mathbb{R}%
^{d\times d}):\text{\eqref{1} admits an exponential dichotomy}\big\}
\end{equation*}
for systems in $\mathbb{R}^d$ with exponential dichotomy, and
\begin{equation*}
\mathrm{BD}^{d}\coloneqq\big\{A\in \mathcal{L}^{\mathrm{Lya}}(\mathbb{N},\mathbb{R}%
^{d\times d}):\text{\eqref{1} admits a Bohl dichotomy}\big\}
\end{equation*}
for systems in $\mathbb{R}^d$ with Bohl dichotomy, we therefore have the inclusion
\begin{equation*}
    \mathrm{ED}^d \subseteq \mathrm{BD}^d.
\end{equation*}
For continuous-time systems the concept of Bohl dichotomy was first proposed in \cite{Bekrayeva2010} where it is called weak dichotomy and the discussion was continued in \cite{BarabanovandBekrayeva2020} by introducing yet another concept of hyperbolicity called almost exponential dichotomy.
In particular, it is shown in \cite{Bekrayeva2010} and \cite{BarabanovandBekrayeva2020} that the concept of Bohl dichotomy is a significant generalization of exponential dichotomy.
A similar discussion in the discrete case has been made in \cite{Barreira2018}, where the discrete analogue of Bohl dichotomy appeared first.

In this paper, we study topological aspects of the sets of systems with exponential and Bohl dichotomy by equipping \(\mathrm{ED}^d\) and \(\mathrm{BD}^d\) with the relative topology inherited from the topological space $\mathcal{L}^{\mathrm{Lya}}(\mathbb{N},\mathbb{R}^{d\times d})$ equipped with the topology of uniform convergence. Our main result is that the interior of \(\mathrm{BD}^d\) equals \(\mathrm{ED}^d\),
\begin{equation*}
    \operatorname{int}\mathrm{BD}^d = \mathrm{ED}^d.
\end{equation*}
For the proof of this result, we study and use properties of Bohl exponents, which are closely related to exponential and Bohl dichotomy. We also apply a perturbation result which is called Millionshikov rotation method and which is used in the continuous case e.g.\ in \cite{BarabanovandBekrayeva2020}.
We provide a detailed analysis of this method in the discrete case in the appendix.
Basic knowledge of the dynamical theory of discrete time systems is helpful, we refer to \cite[pp.\ 335ff]{Elaydi2005} in that regard.

\section{Bohl exponents}

The notion of Bohl spectrum and Bohl exponents was proposed in \cite{Doan2017} for continuous time systems. In this paper we use Bohl exponents to characterize Bohl and exponential dichotomies.
For a detailed analysis of the Bohl exponents and proofs of the following results see \cite{CzornikEtal2023}.
In particular, in \cite[Remark 8]{CzornikEtal2023} a discussion of Bohl exponents in relation to other exponents and equivalent definitions is discussed.
It should be noted that in the literature a series of other but equivalent definitions of these exponents can be found and sometimes they appear under different names (see also \cite[Remark 8]{CzornikEtal2023}).


\begin{definition}[Bohl exponents]
The \emph{upper Bohl exponent} $\overline{\beta }_{A}(L)$ and the \emph{%
lower Bohl exponent }$\underline{\beta }_{A}(L)$ of system \eqref{1} on a subspace $L\subseteq
\mathbb{R}^{d}$, $L\neq \{0\}$, are defined as
\begin{align*}
\overline{\beta }_{A}(L)& :=\inf_{N\in \mathbb{N}}\sup_{\substack{ (n,m)\in
\mathbb{N}\times \mathbb{N}  \\ n-m>N,\text{ }m>N}}\sup \Big\{\frac{1}{n-m}%
\ln \frac{\Vert x(n,x_{0})\Vert }{\Vert x(m,x_{0})\Vert }:x_{0}\in
L\setminus \{0\}\Big\},
\\
\underline{\beta }_{A}(L)& :=\sup_{N\in \mathbb{N}}\inf_{\substack{ (n,m)\in
\mathbb{N}\times \mathbb{N}  \\ n-m>N,\text{ }m>N}}\inf \Big\{\frac{1}{n-m}%
\ln \frac{\Vert x(n,x_{0})\Vert }{\Vert x(m,x_{0})\Vert }:x_{0}\in
L\setminus \{0\}\Big\},
\end{align*}%
and $\overline{\beta }_{A}(\{0\}):=-\infty $, $\underline{\beta }%
_{A}(\{0\}):=+\infty $.
\\[2ex]
\emph{Classical exponents:}
If $\dim L=1$ and $x_{0}\in L\setminus
\{0\}$ we define the notation
\begin{align*}
    \overline\beta_A(x_0) &\coloneqq \overline\beta_A(L) = \inf_{N\in \mathbb{N}}\sup_{\substack{ (n,m)\in
\mathbb{N}\times \mathbb{N}  \\ n-m>N,\text{ }m>N}}\frac{1}{n-m} \ln \frac{\Vert x(n,x_{0})\Vert }{\Vert x(m,x_{0})\Vert },
\\
    \underline\beta_A(x_0) &\coloneqq \underline{\beta }_{A}(L) = \sup_{N\in \mathbb{N}}\inf_{\substack{ (n,m)\in
\mathbb{N}\times \mathbb{N}  \\ n-m>N,\text{ }m>N}} \frac{1}{n-m} \ln \frac{\Vert x(n,x_{0})\Vert }{\Vert x(m,x_{0})\Vert },
\end{align*}
which is independent of the choice of \(x_0 \in L \setminus \{0\}\).
\end{definition}

From the definitions of $\overline{\beta }_{A}(\mathbb{R}^{d})$ and $%
\underline{\beta }_{A}(\mathbb{R}^{d})$ it follows that
\begin{equation*}
\overline{\beta }_{A}(\mathbb{R}^{d})\geq \sup_{x_{0}\in \mathbb{R}%
^{d}\setminus \{0\}}\overline{\beta }_{A}(x_{0})
\end{equation*}%
and%
\begin{equation*}
\underline{\beta }_{A}(\mathbb{R}^{d})\leq \inf_{x_{0}\in \mathbb{R}%
^{d}\setminus \{0\}}\underline{\beta }_{A}(x_{0})
\end{equation*}%
for any $A\in \mathcal{L}^{\mathrm{Lya}}(\mathbb{N},\mathbb{R}^{d\times d}).$
In general the last two inequalities may be strict (see \cite{Czornik2019}%
). In particular, it follows from Lemmas \ref{lem:BD-characterization} and \ref{lem:ED-characterization} below that there are systems with a Bohl dichotomy
which do not admit an exponential dichotomy.

We quote two lemmas and a corollary from  \cite{CzornikEtal2023} which characterize Bohl dichotomy and
exponential dichotomy in terms of Bohl exponents.

\begin{lemma}[Characterization of Bohl dichotomy]\label{lem:BD-characterization}
The following three statements are
equivalent:

(i) System \eqref{1} has a Bohl dichotomy.

(ii) There exists a splitting $L_1 \oplus L_2 = \mathbb{R}^d$ with
\begin{equation*}
\sup_{x_0 \in L_1 \setminus \{0\}} \overline{\beta}_{A}(x_0) < 0 \qquad
\text{and} \qquad \inf_{x_0 \in L_2 \setminus \{0\}} \underline{\beta}%
_{A}(x_0) > 0.
\end{equation*}

(iii) There is $\alpha >0$, such that for all $x_{0}\in \mathbb{R}%
^{d}\setminus \{0\}$,
\begin{equation*}
\overline{\beta }_{A}(x_{0})\leq -\alpha \qquad \text{or}\qquad \underline{%
\beta }_{A}(x_{0})\geq \alpha .
\end{equation*}
Moreover, if system \eqref{1} has a Bohl dichotomy with splitting $L_1 \oplus L_2 = \mathbb R^d$, then statement (ii) holds with that splitting.
\end{lemma}

The negation of Lemma \ref{lem:BD-characterization}\emph{(iii)} yields the following useful criterion for the non-existence of a Bohl dichotomy.

\begin{corollary}[Criterion for non-existence of Bohl dichotomy]
\label{cor:notBD}
System \eqref{1} has no Bohl dichotomy if and only if there exists an $x_{0}\in \mathbb{R}^{d}\setminus \{0\}$ such that
\begin{equation*}
\underline{\beta }_{A}(x_{0})\leq 0\leq \overline{\beta }_{A}(x_{0}).
\end{equation*}%
\end{corollary}

\begin{lemma}[Characterization of exponential dichotomy]\label{lem:ED-characterization}
The following statements are equivalent:

(i) System (\ref{1}) has an exponential dichotomy.

(ii) There exists a splitting $L_{1}\oplus L_{2}=\mathbb{R}^{d}$ with
\begin{equation*}
\overline{\beta }_{A}(L_{1})<0\qquad \text{and}\qquad \underline{\beta }%
_{A}(L_{2})>0.
\end{equation*}
Moreover, if system \eqref{1} has an exponential dichotomy with splitting $L_1 \oplus L_2 = \mathbb R^d$, then statement (ii) holds with that splitting.
\end{lemma}

If system \eqref{1} has an exponential dichotomy then on the associated splitting $L_1 \oplus L_2 = \mathbb{R}^d$ the Bohl exponents have additional uniformity properties. We formulate this result in the special case of a trivial splitting $L_1 \oplus L_2 = \mathbb{R}^d \oplus \{0\}$ or $L_1 \oplus L_2 = \{0\} \oplus \mathbb{R}^d$.

\begin{lemma}[Bohl exponents for trivial exponential dichotomy]
\label{sup_ED}
Suppose that system \eqref{1} has an exponential dichotomy.

(i) If $\sup_{x_{0}\in \mathbb{R}^{d}\setminus \{0\}}\overline{\beta }%
_{A}(x_{0})<0,~$then $\overline{\beta }_{A}(\mathbb{R}^{d})<0$.

(ii) If $\inf_{x_{0}\in \mathbb{R}^{d}\setminus \{0\}}\underline{\beta }%
_{A}(x_{0})>0$, then $\underline{\beta }_{A}(\mathbb{R}^{d})>0$.
\end{lemma}

\begin{proof}
Let $L_1 \oplus L_2 = \mathbb{R}^d$ denote the splitting of the assumed exponential dichotomy of system \eqref{1}.
If $\sup_{x_{0}\in \mathbb{R}^{d}\setminus \{0\}}\overline{\beta }%
_{A}(x_{0})<0$, then
\begin{equation*}
\underset{n\rightarrow \infty }{\lim }\Phi_A (n,0)x_{0}=0
\end{equation*}%
for each $x_{0}\in \mathbb{R}^{d}$. This implies that for subspace $L_{1}$ from
the definition of ED we have $L_{1}=\mathbb{R}^{d}$ and it implies that $%
\overline{\beta }_{A}(\mathbb{R}^{d})<0$. The proof of the second statement
is analogous.
\end{proof}

\section{Applying the Millionshikov rotation method}

In this section we prepare those arguments for the proof of our main result which involve the Millionshikov rotation method.

\begin{lemma}[Exponential growth on subsequence via Bohl exponent]\hfill
\label{lem:double-sequence} Let $B\in \mathcal{L}^{\mathrm{Lya}}(\mathbb{N},%
\mathbb{R}^{k\times k})$ with $\overline{\beta }_{B}(\mathbb{R}^{k})\geq 0$
and $(\varepsilon _{\ell })_{\ell \in \mathbb{N}}$ be a decreasing null
sequence of positive numbers. Then there is $\big((s_{\ell },\tau _{\ell })%
\big)_{\ell \in \mathbb{N}}$ in $\mathbb{N}\times \mathbb{N}$, with
\begin{equation*}
\tau _{0}\geq 2,\qquad \tau _{\ell }<s_{\ell }<\tau _{\ell +1},\qquad
\lim_{\ell \rightarrow \infty }(s_{\ell }-\tau _{\ell })=\infty ,
\end{equation*}%
and
\begin{equation*}
\Vert \Phi _{B}(s_{\ell },\tau _{\ell })\Vert \geq \mathrm{e}^{-\varepsilon
_{\ell }(s_{\ell }-\tau _{\ell })},\qquad \ell \in \mathbb{N}.
\end{equation*}
\end{lemma}

\begin{proof}
For $m,n \in \mathbb N$ with $n - m > 0$, it follows from $\Phi_B(n,0) = \Phi_B(n,m)\Phi_B(m,0)$ and from $\Phi_B(m,0)$ being invertible that
\begin{equation*}
    \sup \bigg\{\frac{1}{n-m}\ln\frac{\Vert\Phi_{B}(n,0)x_0\Vert}{\Vert\Phi_{B}(m,0)x_0\Vert} : x_0 \in \mathbb R^k\setminus\{0\}\bigg\}
    =
    \frac{1}{n-m}\ln \Vert\Phi_{B}(n,m)\Vert.
\end{equation*}
Hence for every $N \in \mathbb N$,
\begin{align*}
    0 \leq \overline\beta_B(\mathbb R^k) \leq \sup_{n-m>N,\,m>N} \frac{1}{n-m}\ln \Vert\Phi_{B}(n,m)\Vert.
\end{align*}
Hence for $\varepsilon > 0$ and every $N \in \mathbb N$, there are $m_{N,\varepsilon}, n_{N,\varepsilon} \in \mathbb N$ with
\begin{align*}
    &n_{N,\varepsilon} - m_{N,\varepsilon} > N, \qquad m_{N,\varepsilon} > N,
    \\
    &\frac{1}{n_{N,\varepsilon}-m_{N,\varepsilon}}\ln \Vert\Phi_{B}(n_{N,\varepsilon},m_{N,\varepsilon})\Vert > -\varepsilon.
\end{align*}
We define the sequences $\big((s_\ell,\tau_\ell)\big)_{\ell\in\mathbb N}$ recursively by setting
\begin{equation*}
    \tau_0 \coloneqq m_{2,\varepsilon_0}, \qquad s_0 \coloneqq n_{2,\varepsilon_0},
\end{equation*}
and for $\ell \in \mathbb N$ with $\ell \geq 1$ by setting
\begin{equation*}
    \tau_\ell \coloneqq m_{s_{\ell-1}+1,\varepsilon_\ell}, \qquad s_\ell \coloneqq n_{s_{\ell-1}+1,\varepsilon_\ell}.
    \qedhere
\end{equation*}
\end{proof}

\begin{lemma}[Exponential decay on subsequence via Bohl exponent]\hfill
\label{lem:double-sequence_lower} Let $B\in \mathcal{L}^{\mathrm{Lya}}(%
\mathbb{N},\mathbb{R}^{k\times k})$ with $\underline{\beta }_{B}(\mathbb{R}%
^{k})\leq -\delta <0$ and $(\varepsilon _{\ell })_{\ell \in \mathbb{N}}$ be
a decreasing null sequence of positive numbers. Then there is $\big((s_{\ell
},\tau _{\ell })\big)_{\ell \in \mathbb{N}}$ in $\mathbb{N}\times \mathbb{N}$%
, with
\begin{equation}
\tau _{0}\geq 2,\qquad \tau _{\ell }<s_{\ell }<\tau _{\ell +1},\qquad
\lim_{\ell \rightarrow \infty }(s_{\ell }-\tau _{\ell })=\infty ,\qquad \ell
\in \mathbb{N}, \label{_0}
\end{equation}
\begin{equation}
   \frac{1}{s_\ell-\tau_\ell} \ln\left(\frac 2{\sin\varepsilon
_\ell}\right) <\varepsilon_\ell,\qquad \ell \in \mathbb N, \label{_1}
\end{equation}
and
\begin{equation}
\left\Vert \Phi _{B}(\tau _{\ell },s_{\ell })\right\Vert ^{-1}\leq
\mathrm{e}^{\left( -\delta +\varepsilon _{\ell }\right) (s_{\ell }-\tau
_{\ell })},\qquad \ell \in \mathbb{N}. \label{_2}
\end{equation}
\end{lemma}

\begin{proof}
The proof is similar to Lemma \ref{lem:double-sequence}, using that for $m,n \in \mathbb N$, $n - m \geq 1$,
\begin{align*}
    &\inf \bigg\{\frac{1}{n-m}\ln\frac{\Vert\Phi_{B}(n,0)x_0\Vert}{\Vert\Phi_{B}(m,0)x_0\Vert} : x_0 \in \mathbb R^k\setminus\{0\}\bigg\}
    \\
    &\quad =
    \inf \bigg\{-\frac{1}{n-m}\ln\frac{\Vert\Phi_{B}(m,0)x_0\Vert}{\Vert\Phi_{B}(n,0)x_0\Vert} : x_0 \in \mathbb R^k\setminus\{0\}\bigg\}
    \\
    &\quad =
    -\sup \bigg\{\frac{1}{n-m}\ln\frac{\Vert\Phi_{B}(m,n)\Phi_B(n,0)x_0\Vert}{\Vert\Phi_B(n,0)x_0\Vert} : x_0 \in \mathbb R^k\setminus\{0\}\bigg\}
    \\
    &\quad = \frac{1}{n-m}\ln\Vert\Phi_{B}(m,n)\Vert^{-1}.
\end{align*}
As a consequence, for every $N \in \mathbb N$,
\begin{align*}
    0 < -\delta \leq \underline\beta_B(\mathbb R^k) 
    \leq \inf_{n-m>N,\,m>N} \frac{1}{n-m}\ln \Vert\Phi_{B}(m,n)\Vert^{-1}.
\end{align*}
Hence for \(\varepsilon>0\) and every \(N \in \mathbb N\), there are \(m_{N,\varepsilon},n_{N,\varepsilon} \in \mathbb N\) with
\begin{align*}
    n_{N,\varepsilon} - m_{N,\varepsilon} > \max\bigg\{N,\frac{\ln\big(\frac 2{\sin\varepsilon}\big)}{\varepsilon}\bigg\}, \qquad m_{N,\varepsilon} > N,
    \\
    \frac 1{n_{N,\varepsilon} - m_{N,\varepsilon}} \ln\Vert\Phi_B(m_{N,\varepsilon}, n_{N,\varepsilon})\Vert^{-1} < -\delta + \varepsilon.
\end{align*}
We define the sequences \((s_\ell,\tau_\ell)_{\ell\in\mathbb N}\) recursively by setting
\begin{equation*}
    \tau_0 \coloneqq m_{2,\varepsilon_0},
    \qquad
    s_0 \coloneqq n_{2,\varepsilon_0},
\end{equation*}
and for \(\ell \in \mathbb N\) with \(\ell \geq 1\) by setting
\begin{equation*}
    \tau_\ell \coloneqq m_{s_{\ell-1}+1,\varepsilon_\ell},
    \qquad
    s_\ell \coloneqq n_{s_{\ell-1}+1,\varepsilon_\ell}.
    \qedhere
\end{equation*}
\end{proof}

The following two lemmas assume conditions for upper Bohl exponents assuring the existence of a perturbed system that has a solution with specific Bohl exponents.

We make the following observation first though:

\begin{remark}[\(\mathcal{L}^{\mathrm{Lya}}(\mathbb{N},\mathbb{R}^{k\times k})\) is open]\label{rem:Lya_open}
The set $\mathcal{L}^{\mathrm{Lya}}(\mathbb{N},\mathbb{R}^{k\times k})$ is an open subset of $\mathcal{L}^{\mathrm{\infty }}(\mathbb{N},\mathbb{R}^{k\times k})$.
Indeed, this can be proved for $B \in \mathcal{L}^{\mathrm{Lya}}(\mathbb{N},\mathbb{R}^{k\times k})$, $B' \in \mathcal{L}^{\mathrm{\infty }}(\mathbb{N},\mathbb{R}^{k\times k})$, with $\Vert B - B'\Vert_\infty$ is sufficiently small by
\begin{equation*}
    B'(n) = B(n)\big(I - B(n)^{-1}(B(n) - B'(n))\big),
    \qquad n \in \mathbb N
\end{equation*}
and representing the inverse of $I - B(n)^{-1}(B(n) - B'(n))$ by the Neumann series.
\end{remark}

\begin{lemma}[Perturbation for special solution I]
\label{lem:lya_perturbed}
Let $z_{0}\in \mathbb{R}^{k}\setminus \{0\}$ and $%
B\in \mathcal{L}^{\mathrm{Lya}}(\mathbb{N},\mathbb{R}^{k\times k}),$ $k\geq 2
$ with
\begin{equation*}
\sup_{x_{0}\in \mathbb{R}^{k}\setminus \{0\}}\overline{\beta }%
_{B}(x_{0})<0\qquad \text{and}\qquad \overline{\beta }_{B}(\mathbb{R}%
^{k})\geq 0.
\end{equation*}
Then there exists a $Q\in \mathcal{L}^{\infty }(\mathbb{N},\mathbb{R}%
^{k\times k})$ with

(i) $\lim_{\ell \rightarrow \infty }Q(\ell )=0$,

(ii) $B+Q\in \mathcal{L}^{\mathrm{Lya}}(\mathbb{N},\mathbb{R}^{k\times k})$,

(iii) $\underline{\beta }_{B+Q}(z_{0})<0$ and $\overline{\beta }%
_{B+Q}(z_{0})\geq 0$.
\end{lemma}

To prove Lemma \ref{lem:lya_perturbed} we apply the Millionshikov Rotation Method in its algebraic form formulated in Remark \ref{rem:millionshikov_alg}(a).
In the proof of Lemma \ref{Lem:10} we will again apply the rotation method, but it is then more convenient to use the equivalent formulation of the rotation method given by Theorem \ref{MRMC1}.

\begin{proof}
We construct recursively a strictly increasing sequence $(T_{j})_{j\in
\mathbb{N}}$ in $\mathbb N$ and $Q\in \mathcal{L}^{\infty }(\mathbb{N},\mathbb{R}^{k\times
k})$ on $[T_{j},T_{j+1}-1]$ for $j\in \mathbb{N}$.
Let $(y(n,y_{0}))_{n\in\mathbb N}$
denote the solution of $y(n+1)=B(n)y(n)$, $n\in \mathbb{N}$, $y(0)=y_{0}$.

Let us fix $-\alpha \in \left( \sup_{x_{0}\in \mathbb{R}^{k}\setminus
\{0\}}\overline{\beta }_{B}(x_{0}),0\right) $.
Then for each $y_{0}\in
\mathbb{R}^{k}$, $\varepsilon >0$ there exists $N(\varepsilon ,y_{0})\in
\mathbb{N}$ such that for all $n,$ $m\in \mathbb{N}$, $n-m>N(\varepsilon
,y_{0})$ we have
\begin{equation}
\frac{\Vert y(n,y_{0})\Vert }{\Vert y(m,y_{0})\Vert }\leq \mathrm{e}^{\left(
-\alpha +\varepsilon \right) (n-m)}.  \label{eq:uniform_new}
\end{equation}

Since \(\mathcal{L}^{\mathrm{Lya}}(\mathbb{N},\mathbb{R}^{k\times k})\) is open by Remark \ref{rem:Lya_open}, let $\varepsilon' > 0$ be, such that $\Vert B - B'\Vert_\infty \leq \varepsilon'$ implies $B' \in \mathcal{L}^{\mathrm{Lya}}(\mathbb{N},\mathbb{R}^{k\times k})$ for $B' \in \mathcal{L}^{\mathrm{\infty }}(\mathbb{N},\mathbb{R}^{k\times k})$.
Let $b:=\max \{\Vert B\Vert _{\infty },\Vert B^{-1}\Vert _{\infty }\}$ and $%
\varepsilon _{\ell }:=\min \left\{\frac{1}{\ell +1}%
,\frac{\varepsilon^\prime}b\right\} $ for $\ell \in \mathbb{N}$. Using the
assumption $\overline{\beta }_{B}(\mathbb{R}^{k})\geq 0$, Lemma \ref%
{lem:double-sequence} yields a sequence $\big((s_{\ell },\tau _{\ell })\big)%
_{\ell \in \mathbb{N}}$ in $\mathbb{N}\times \mathbb{N}$, with $\tau _{0}>2$%
, $\tau _{\ell }<s_{\ell }<\tau _{\ell +1}$, $\lim_{\ell \rightarrow \infty
}(s_{\ell }-\tau _{\ell })=\infty $ and
\begin{equation}
\Vert \Phi _{B}(s_{\ell },\tau _{\ell })\Vert \geq \mathrm{e}^{-\varepsilon
_{\ell }(s_{\ell }-\tau _{\ell })},\qquad \ell \in \mathbb{N}.
\label{eq:transition_B_eps}
\end{equation}%
We define $T_{0}:=0$, $T_{1}:=1$ and $Q(\ell ):=0$ for $\ell \in
\lbrack T_{0},T_{1}-1]=\{0\}$. For $j\in \mathbb{N}$ with $j\geq 1$ assume
that $T_{0},\dots ,T_{j}$ and $Q(0),\dots ,Q(T_{j}-1)$ are defined. To
define $T_{j+1}$ and $Q(\ell )$ for $\ell \in \lbrack T_{j},T_{j+1}-1]$ we
distinguish case (a) $j$ is odd and case (b) $j$ is even.

Case (a) $j$ is odd.
Define
\begin{equation*}
v:=\big(B(T_{j}-1)+Q(T_{j}-1)\big)\cdots \big(B(0)+Q(0)\big)z_{0}\quad \text{%
and}\quad y_{0}:=\Phi _{B}(0,T_{j})v.
\end{equation*}%
Using \eqref{eq:uniform_new} for $y_{0}$ and $\varepsilon
=\varepsilon _{j}$, there exist $\rho _{j},$ $\sigma _{j}\in
\mathbb{N}$, $\rho _{j} > \sigma_j \geq T_j$, $\rho_j -\sigma _{j}>j$ such that
\begin{equation}
\frac{\Vert y(\rho _{j},y_{0})\Vert }{\Vert y(\sigma _{j},y_{0})\Vert }\leq
\mathrm{e}^{\left( -\alpha +\varepsilon _{j}\right) (\rho _{j}-\sigma _{j})}.
\label{eq:odd_perturbed}
\end{equation}
We set $T_{j+1}:=\rho _{j}$ and $Q(\ell ):=0$ for $\ell \in \lbrack
T_{j},T_{j+1}-1]$.

Case (b) $j$ is even.
Since $\lim_{\ell \rightarrow \infty
}\tau _{\ell }=\infty $ and $\lim_{\ell \rightarrow \infty }(s_{\ell }-\tau
_{\ell })=\infty $ there exists an $\ell _{j}\in \mathbb{N}$ with
\begin{equation*}
    \tau_{\ell _{j}}
    \geq
    T_{j}+2
\end{equation*}
and, since \(\lim_{\ell\to\infty}\mathrm e^{-\varepsilon_j(s_{\ell}-\tau_{\ell})} = 0\), with
\begin{equation}
\frac{\sin \varepsilon _{j}}{2}\geq \mathrm{e}^{-\varepsilon _{j}(s_{\ell
_{j}}-\tau _{\ell _{j}})}.  \label{eq:sin_eps}
\end{equation}%
We set $T_{j+1}:=s_{\ell _{j}}$ and $Q(\ell ):=0$ for $\ell \in \lbrack
T_{j},\tau _{\ell _{j}}-2]$. Applying the Millionshikov rotation method
Remark \ref{rem:millionshikov_alg}(a) with $\varepsilon =\varepsilon _{j}$, $m=\tau
_{\ell _{j}}-1$, $n=s_{\ell _{j}}-1$, and
\begin{equation*}
v=\big(B(\tau _{\ell _{j}}-2)+Q(\tau _{\ell _{j}}-2)\big)\cdots \big(%
B(0)+Q(0)\big)z_{0},
\end{equation*}%
yields an $R\in \mathbb{R}^{d\times d}$ with $\Vert R\Vert \leq \varepsilon
_{j}b$, $B(\tau _{\ell _{j}}-1)+R\in \mathrm{GL}(k)$ and
\begin{align*}
\Vert & B(s_{\ell _{j}}-1)\cdots B(\tau _{\ell _{j}})\big(B(\tau _{\ell
_{j}}-1)+R\big)v\Vert \\
& \geq \frac{\sin \varepsilon _{j}}{2}\Vert B(s_{\ell _{j}}-1)\cdots B(\tau
_{\ell _{j}})\Vert \cdot \Vert \big(B(\tau _{\ell _{j}}-1)+R\big)v\Vert .
\end{align*}%
Dividing and using \eqref{eq:transition_B_eps} and \eqref{eq:sin_eps} we
obtain
\begin{equation}
\begin{split}
\frac{\Vert B(s_{\ell _{j}}-1)\cdots B(\tau _{\ell _{j}})\big(B(\tau _{\ell
_{j}}-1)+R\big)v\Vert }{\Vert \big(B(\tau _{\ell _{j}}-1)+R)\big)v\Vert }&
\geq \frac{\sin \varepsilon _{j}}{2}\Vert B(s_{\ell _{j}}-1)\cdots B(\tau
_{\ell _{j}})\Vert \\
& \geq \mathrm{e}^{-\varepsilon _{j}(s_{\ell _{j}}-\tau _{\ell _{j}})}\cdot
\mathrm{e}^{-\varepsilon _{\ell _{j}}(s_{\ell _{j}}-\tau _{\ell _{j}})} \\
& =\mathrm{e}^{-(\varepsilon _{j}+\varepsilon _{\ell _{j}})(s_{\ell
_{j}}-\tau _{\ell _{j}})}.
\end{split}
\label{eq:even_perturbed}
\end{equation}%
We set
\begin{equation*}
Q(\ell ):=%
\begin{cases}
R, & \ell =\tau _{\ell _{j}}-1, \\
0, & \ell \in \lbrack \tau _{\ell _{j}},T_{j+1}-1],%
\end{cases}%
\end{equation*}%
and rewrite \eqref{eq:even_perturbed} as
\begin{equation}
\frac{\Vert \big(B(s_{\ell _{j}}-1)+Q(s_{\ell _{j}}-1)\big)\cdots \big(%
B(0)+Q(0)\big)z_{0}\Vert }{\Vert \big(B(\tau _{\ell _{j}}-1)+Q(\tau _{\ell
_{j}}-1)\big)\cdots \big(B(0)+Q(0)\big)z_{0}\Vert }\geq \mathrm{e}%
^{-(\varepsilon _{j}+\varepsilon _{\ell _{j}})(s_{\ell _{j}}-\tau _{\ell
_{j}})},  \label{eq:even_perturbed2}
\end{equation}%
which ends the discussion of case (b) and the recursive definition of $Q\in
\mathcal{L}^{\infty }(\mathbb{N},\mathbb{R}^{k\times k})$.

For $j\in \mathbb{N}$ and $\ell \in \lbrack T_{j},T_{j+1}-1]$, either $Q(\ell) = 0$ or $\Vert Q(\ell)\Vert \leq \varepsilon_jb$, wich implies $\lim_{\ell\to\infty}Q(\ell) = 0$ and $\Vert Q\Vert \leq \varepsilon'$, wich implies $B + Q \in \mathcal L^{\mathrm{Lya}}(\mathbb N,\mathbb R^{d\times d})$.
By \eqref{eq:odd_perturbed} $\underline{\beta }_{B+Q}(z_{0})\leq -\alpha <0$ .
By \eqref{eq:even_perturbed2} $\overline{\beta }_{B+Q}(z_{0})\geq 0$.
\end{proof}

In a similar way as we have proved Lemma \ref{lem:lya_perturbed}, the following Lemma \ref{lem:lya_perturbed_v2} can be proved.

\begin{lemma}[Perturbation for special solution II]
\label{lem:lya_perturbed_v2} Let $z_{0}\in \mathbb{R}^{k}\setminus \{0\}$
and $B\in \mathcal{L}^{\mathrm{Lya}}(\mathbb{N},\mathbb{R}^{k\times k}),$ $%
k\geq 2$ with
\begin{equation*}
\sup_{x_{0}\in \mathbb{R}^{k}\setminus \{0\}}\overline{\beta }%
_{B}(x_{0})\leq 0\qquad \text{and}\qquad \overline{\beta }_{B}(\mathbb{R}%
^{k})\geq 0.
\end{equation*}%
Then there exists a $Q\in \mathcal{L}^{\infty }(\mathbb{N},\mathbb{R}%
^{k\times k})$ with

(i) $\lim_{\ell \rightarrow \infty }Q(\ell )=0$,

(ii) $B+Q\in \mathcal{L}^{\mathrm{Lya}}(\mathbb{N},\mathbb{R}^{k\times k})$,

(iii) $\underline{\beta }_{B+Q}(z_{0})\leq 0$ and $\overline{\beta }%
_{B+Q}(z_{0})\geq 0$.
\end{lemma}

The following observation will be useful.

\begin{remark}[Adjusting the norm of the perturbation]
\label{Rem:eps}
Let $B\in \mathcal{L}^{\mathrm{Lya}}(\mathbb{N},\mathbb{R}^{k\times k})$ and $\varepsilon' > 0$, such that $B + Q \in \mathcal{L}^{\mathrm{Lya}}(\mathbb{N},\mathbb{R}^{k\times k})$ if $\Vert Q\Vert_\infty < \varepsilon'$.
Then for $Q \in \mathcal{L}^{\infty }(\mathbb{N},\mathbb{R}^{k\times k})$ with $\Vert Q\Vert < \varepsilon'$ and $\lim_{\ell \rightarrow \infty }Q(\ell )=0$, and for any subspace $L$ of $\mathbb{R}^{k}$ and any $\varepsilon \in (0,\varepsilon')$, we have
\begin{equation*}
\overline{\beta }_{B+Q}(L)=\overline{\beta }_{B+Q_{\varepsilon }}(L)\text{
and }\underline{\beta }_{B+Q}(L)=\underline{\beta }_{B+Q_{\varepsilon }}(L),
\end{equation*}%
where $Q_{\varepsilon }$ is defined for $n \in \mathbb N$ by
\begin{equation*}
Q_{\varepsilon }(n)=\left\{
\begin{array}{l}
Q(n)\text{ if }\left\Vert Q(n)\right\Vert \leq \varepsilon,  \\
0\text{ otherwise.}%
\end{array}%
\right.
\end{equation*}%
This holds because the sequences $Q$ and $Q_{\varepsilon }$ differ only for
finitely many $n$. Also note that $\left\Vert Q_{\varepsilon}\right\Vert \leq \varepsilon$.
\end{remark}

Using Remark \ref{Rem:eps} and Corollary \ref{cor:notBD} we get from Lemmas %
\ref{lem:lya_perturbed} and \ref{lem:lya_perturbed_v2} the following corollary:

\begin{corollary}[Perturbation for special solution]
\label{C1}Under the assumption of Lemma \ref{lem:lya_perturbed} (Lemma \ref%
{lem:lya_perturbed_v2})  for each $\varepsilon >0$ and $z_0 \in \mathbb R^k \setminus \{0\}$ there exists a $%
Q_{\varepsilon }\in \mathcal{L}^{\infty }(\mathbb{N},\mathbb{R}^{k\times k})$
such that $\left\Vert Q_{\varepsilon }\right\Vert _{\infty }\leq \varepsilon
$, $B+Q_{\varepsilon }\in \mathcal{L}^{\mathrm{Lya}}(\mathbb{N},\mathbb{R}%
^{k\times k})$ and
\begin{equation*}
\underline{\beta }_{B+Q_{\varepsilon }}(z_{0})<0
\quad \text{and} \quad
\overline{\beta }%
_{B+Q_{\varepsilon }}(z_{0})\geq 0,
\end{equation*}%
\begin{equation*}
\big(
\underline{\beta }
_{B+Q_{\varepsilon }}(z_{0})\leq 0
\quad \text{and} \quad
\overline{\beta }%
_{B+Q_{\varepsilon }}(z_{0})\geq 0
\big).
\end{equation*}
In both cases $B+Q_{\varepsilon }\notin \mathrm{BD}^{k}.$
\end{corollary}

The following two lemmas assume conditions for lower Bohl exponents assuring the existence of a perturbed system that has a solution with specific Bohl exponents.

\begin{lemma}[Perturbation with special solution I]
\label{lem:lya_perturbed_lower_v2_step1} Let $B\in \mathcal{L}^{\mathrm{Lya}%
}(\mathbb{N},\mathbb{R}^{k\times k})$ with
\begin{equation*}
\inf_{x_{0}\in \mathbb{R}^{k}\setminus \{0\}}\underline{\beta }%
_{B}(x_{0})>0\qquad \text{and}\qquad \underline{\beta }_{B}(\mathbb{R}%
^{k})\leq 0.
\end{equation*}%
Then for any $\varepsilon >0$ there exists a $Q\in \mathcal{L}^{\infty }(%
\mathbb{N},\mathbb{R}^{k\times k})$ with

(i) $\left\Vert Q\right\Vert _{\infty }<\varepsilon $,

(ii) $B+Q\in \mathcal{L}^{\mathrm{Lya}}(\mathbb{N},\mathbb{R}^{k\times k})$,

(iii) $\inf_{x_{0}\in \mathbb{R}^{k}\setminus \{0\}}\underline{\beta }%
_{B+Q}(x_{0})>0\qquad $

(iv) $\underline{\beta }_{B+Q}(\mathbb{R}^{k})<0$.
\end{lemma}

\begin{proof}
Let us denote
\begin{equation*}
\inf_{x_{0}\in \mathbb{R}^{k}\setminus \{0\}}\underline{\beta }%
_{B}(x_{0})=\nu >0
\end{equation*}%
and fix $\varepsilon >0.$
Since $\mathcal{L}^{\mathrm{Lya}}(\mathbb{N},\mathbb{R}^{k\times k})$ is open by Remark \ref{rem:Lya_open}, there is \(\varepsilon' > 0\), such that for all $Q \in \mathcal{L}^{\mathrm{Lya}}(\mathbb{N},\mathbb{R}^{k\times k})$ with $\Vert Q\Vert_\infty \leq \varepsilon'$ we have $B+Q \in \mathcal{L}^{\mathrm{Lya}}(\mathbb{N},\mathbb{R}^{k\times k})$.
It suffices to prove the statement under the assumption $\varepsilon < \min(\Vert B\Vert_\infty,\varepsilon')$.
Let us take
\begin{equation*}
\delta \in \left( 0,\min \left\{ \nu ,-\ln \left( 1-\frac{\varepsilon }{%
\left\Vert B\right\Vert _{\infty }}\right) \right\} \right) ,
\end{equation*}%
then
\begin{equation*}
\left\vert 1-\mathrm e^{-\delta }\right\vert \leq \frac{\varepsilon }{\left\Vert
B\right\Vert _{\infty }}
\end{equation*}%
and therefore for
\begin{equation*}
Q(n)=B(n)\left(\mathrm e^{-\delta }-1\right).
\end{equation*}%
we have%
\begin{equation*}
\left\Vert Q\right\Vert _{\infty }\leq \varepsilon .
\end{equation*}%
Moreover, we have%
\begin{equation*}
B(n)+Q(n)=B(n)\mathrm  e^{-\delta }.
\end{equation*}%
The last relation implies%
\begin{equation*}
\underline{\beta }_{B+Q}(x_{0})=\underline{\beta }_{B}(x_{0})-\delta > \nu - \delta,
\end{equation*}%
for any $x_{0}\in \mathbb{R}^{k}\setminus \{0\}$ and
\begin{equation*}
\underline{\beta }_{B+Q}(\mathbb{R}^{k})=\underline{\beta }_{B}(\mathbb{R}%
^{k})-\delta.
\end{equation*}
Therefore
\begin{equation*}
\inf_{x_{0}\in \mathbb{R}^{k}\setminus \{0\}}\underline{\beta }%
_{B+Q}(x_{0})>0
\qquad\text{and}\qquad
\underline{\beta }_{B+Q}(\mathbb{R}^{k})\leq
-\delta \text{.}
\end{equation*}
\end{proof}

\begin{lemma}[Perturbation with special solution II]
\label{Lem:10}
Let $B \in \mathcal{L}^{\mathrm{Lya}}(\mathbb{N},\mathbb{R}^{k\times k})$ with $\underline{\beta}_{B}(\mathbb{R}^{k})<0$ and $k \geq 2$.
Then there exists a $Q\in \mathcal{L}^{\infty }(\mathbb{N},\mathbb{R}^{k\times k})$ with

(i) $\underset{\ell\rightarrow \infty }{\lim }Q(\ell)=0$,

(ii) $B+Q\in \mathcal{L}^{\mathrm{Lya}}(\mathbb{N},\mathbb{R}^{k\times k})$,

(iii) $\underset{x_{0}\in \mathbb{R}^{k}\setminus \{0\}}{\inf }\underline{%
\beta }_{B+Q}(x_{0})<0.$
\end{lemma}

In the proof of Lemma \ref{Lem:10} we apply the Millionshikov Rotation Method as formulated in Theorem \ref{MRMC1}.

\begin{proof}
Let $\varepsilon ^{\prime }>0$ such that $B^{\prime }\in \mathcal{L}^{\mathrm{Lya}}(\mathbb{N},\mathbb{R}^{k\times k})$ for all $B^{\prime }\in
\mathcal{L}^{\mathrm{\infty }}(\mathbb{N},\mathbb{R}^{k\times k}),$ $%
\left\Vert B-B^{\prime }\right\Vert_{\infty } <\varepsilon ^{\prime }$\emph{.} Let $%
b:=\max \{\Vert B\Vert _{\infty },\Vert B^{-1}\Vert _{\infty }\}$ and $%
\varepsilon _{\ell }:=\min \left\{\frac{1}{\ell+1}%
,\varepsilon ^{\prime }\right\} $ for $\ell \in \mathbb{N}$. Using the
assumption \underline{$\beta $}$_{B}(\mathbb{R}^{k})<0$, Lemma \ref%
{lem:double-sequence_lower} with any $-\delta \in \left( \underline{\beta }%
_{B}(\mathbb{R}^{k}),0\right) $ yields a sequence $\big((s_{\ell },\tau
_{\ell })\big)_{\ell \in \mathbb{N}}$ in $\mathbb{N}\times \mathbb{N}$ such
that \eqref{_0}, \eqref{_1} and \eqref{_2} are satisfied.

First, for each $\ell \in \mathbb{N}$ we will construct a perturbation $Q_{\ell } \in \mathcal L^\infty(\mathbb N,\mathbb R^{k\times k})$ and an initial condition $z_{\ell ,0}$ for the system
\begin{equation}
z(n+1)=(B(n)+Q_{\ell }(n))z(n),\quad \quad   \label{Perturbed_Bohl_blr10}
\end{equation}%
such that

\begin{enumerate}
\item $Q_{\ell }(n)=0$ for $n\neq s_{j},$ $j\in \mathbb{N}$,

\item $\Vert Q_{\ell }(s_{j})\Vert <\varepsilon _{j}b$ for $j\in \mathbb{N},$

\item the solution $(z_{\ell }(n,z_{\ell ,0}))_{n\in \mathbb{N}}$
of (\ref{Perturbed_Bohl_blr10}), satisfies
\begin{equation}
\Vert z_{\ell }(\tau _{q},z_{\ell ,0})\Vert \geq \frac{\sin \varepsilon
_{q}}{2}\;\Vert \Phi _{B}(\tau _{q},s_{q})\Vert \Vert z_{\ell }(s_{q},z_{\ell
,0})\Vert ,\qquad q=1,\ldots ,\ell .  \label{Perturbed_Bohl_blr11}
\end{equation}
\end{enumerate}

The perturbation and the initial value will be obtained as
\begin{equation}\label{eq:def_Q_l}
   Q_{\ell} \coloneqq Q_\ell^{(\ell)} + \dots + Q_\ell^{(1)} \text{ and } z_{\ell,0} \coloneqq \frac{z_{\ell,0}^{(1)}}{\Vert z_{\ell,0}^{(1)}\Vert}
\end{equation}
where $Q_\ell^{(j)}$ and $z_{\ell,0}^{(j)}$ are constructed for $j = \ell, \ell-1, \dots, 1$ such that

\begin{enumerate}
\item[(i)] $Q_{\ell }^{(j)}(n)=0$ for $n\neq s_{j},$ $j\in \mathbb{N}$,

\item[(ii)] $\Vert Q_{\ell }^{(j)}(s_{j})\Vert <\varepsilon _{j}b$ for $j\in \mathbb{N},$

\item[(iii)] the solution $(z_{\ell }^{(j)}(n,z_{\ell ,0}^{(j)}))_{n\in \mathbb{N}}$ of the system
\begin{equation*}
    z(n+1) = \big(B(n) + Q_\ell^{(\ell)}(n) + \dots + Q_\ell^{(j)}(n)\big)z(n),
\end{equation*}
satisfies
\begin{equation}
   \Vert z_\ell^{(j)}(\tau_q,z_{\ell,0}^{(j)})\Vert \geq \frac{\sin \varepsilon_q}2 \Vert \Phi_B(\tau_q,s_q)\Vert \Vert z_\ell(s_q,z_{\ell,0}^{(j)})\Vert,
   \qquad q=j,\dots,\ell.
   \label{Perturbed_Bohl_jblr11}
\end{equation}
\end{enumerate}

Applying the Millionshikov backward rotation method Theorem \ref{MRMC1}(b) to the sequence $A=B$ with
\begin{equation*}
   \varepsilon = \varepsilon_\ell,\, x_{0} = z_{0},\, k=\tau _\ell
   \qquad\text{and}\qquad m=s_{\ell },
\end{equation*}
we obtain a perturbation $Q_\ell^{(\ell)}$ and an initial value $z_{\ell,0}^{(\ell)}$ with the properties (i), (ii) and (iii).

$j+1\rightarrow j$:
Suppose now that we have constructed a perturbation $Q_\ell^{(j+1)}$ and an initial value $z_{\ell,0}^{(j+1)}$ satisfying (i), (ii) and (iii).
Applying the Millionshikov backward rotation method Theorem \ref{MRMC1}(b) to the sequence $A = B + Q_\ell^{(\ell)} + \dots + Q_\ell^{(j+1)}$ with
\begin{equation*}
   \varepsilon = \varepsilon_{j},\,
   x_{0} = z_{\ell,0}^{(j+1)},\,
   k = \tau_{j}\qquad
   \text{and}\qquad m = s_{j},
\end{equation*}
we obtain a sequence $Q_\ell^{(j)}$ and an initial value $z_{\ell,0}^{(j)}$ such that
\begin{equation}\label{eq:eq_sol_Mil_bcw}
   z_\ell^{(j)}(s_j+n,z_{\ell,0}^{(j)}) = z_\ell^{(j+1)}(s_j+n,z_{\ell,0}^{(j+1)}),\qquad n\geq 1
\end{equation}
and such that
\begin{equation*}
   \Vert z_\ell^{(j)}(\tau_{s_j},z_{\ell,0}^{(j)})
   \geq \frac 12 \sin\varepsilon_j \Vert \Phi_{B+Q_\ell^{(\ell)}+\dots+Q_\ell^{(j+1)}}(\tau_j,s_j)\Vert \Vert z_\ell^{(j)}(s_j,z_{\ell,0}^{(j)})\Vert.
\end{equation*}
That $Q_\ell^{(j)}$ satisfies (i) and (ii) follows from the Millionshikov rotation method.
We see that (iii) holds for $q = j+1,\dots, \ell$ from \eqref{eq:eq_sol_Mil_bcw}.
To see that (iii) holds for $q = j$ we note that $B(n) + Q_\ell^{(\ell)}(n) + \dots+Q_\ell^{(j)}(n) = B(n)$ for $n = 0,\dots,s_j-1$ and hence
\begin{equation*}
   \Phi_{B + Q_\ell^{(\ell)} + \dots + Q_\ell^{(j+1)}}(s_{j},\tau_{j}) = \Phi_{B}(s_{j},\tau_{j}).
\end{equation*}
We now define $Q_\ell$ and $z_{\ell,0}$ according to \eqref{eq:def_Q_l} and observe that 1., 2. and 3. hold.


Since the set $\{z_{\ell}(s_{1},z_{\ell,0}) \colon \ell \in \mathbb{N}\}$ is bounded as a consequence of the assumption $\left\Vert z_{\ell,0}\right\Vert =1$ for all $\ell \in \mathbb{N}$, and the set $\{Q_{\ell }\left( s_{1}\right) \colon \ell \in \mathbb{N}\}$ is bounded, there exists a sequence $(\ell_{j}^{(1)})_{j\in \mathbb{N}}$ of natural numbers such that the sequences
\begin{equation*}
   \left( z_{\ell_{j}^{(1)}}\left( s_{1},z_{\ell_{j}^{(1)},0}\right) \right) _{j\in\mathbb{N}}
   \text{ and }\left( Q_{\ell_{j}^{(1)}}(s_{1})\right) _{j\in \mathbb{N}}
\end{equation*}
are convergent.
Denote
\begin{equation*}
   v(s_{1}) = \underset{j\rightarrow \infty }{\lim }z_{\ell_{j}^{(1)}}\left( s_{1},z_{\ell_{j}^{(1)},0}\right)
   \text{ and }Q(s_{1}) = \underset{j\rightarrow \infty }{\lim }Q_{\ell_{j}^{(1)}}(s_{1}).
\end{equation*}
From the sequence $(\ell_{j}^{(1)})_{j\in \mathbb{N}}$ we choose a subsequence $(\ell_{j}^{(2)})_{j\in \mathbb{N}}$ such that the sequences
\begin{equation*}
   \left(z_{\ell_{j}^{(2)}}\left( s_{2},z_{\ell_{j}^{(2)},0}\right)\right)_{j\in\mathbb{N}}
   \text{ and }\left( Q_{\ell_{j}^{(2)}}(s_{2})\right)_{j\in \mathbb{N}}
\end{equation*}
are convergent and we denote
\begin{equation*}
   v(s_{2}) = \underset{j\rightarrow \infty }{\lim } z_{\ell_{j}^{(2)}}\left( s_{2},z_{\ell_{j}^{(2)},0}\right)
   \text{ and }Q(s_{2}) = \underset{j\rightarrow \infty }{\lim }Q_{\ell_{j}^{(2)}}(s_{2}).
\end{equation*}
We will continue this procedure for all $s_{i},$ $i\in \mathbb{N}$.
In this way we obtain sequences $\left(v(s_{i})\right)_{i\in\mathbb{N}}$ and $\left(Q(s_{i})\right) _{i \in \mathbb{N}}$.
We extend these sequences to sequences $\left(v(n)\right)_{n\in\mathbb{N}}$ and $\left(Q(n)\right)_{n\in \mathbb{N}}$ as follows
\begin{equation*}
   Q(n) =
   \begin{cases}
      &Q(s_{i})\text{ if }n=s_{i}\text{ for certain }i\in \mathbb{N},
      \\
      &0\text{ otherwise},
   \end{cases}
\end{equation*}
and
\begin{equation*}
   v(n) =
   \begin{cases}
      &B^{-1}\left( n\right) \dots B^{-1}\left( s_{1}-1\right) v(s_{1})\text{ for }n\in \left[ 0,s_{1}-1\right],
      \\
      &v(s_{i})\text{ if }n=s_{i}\text{ for certain }i\in \mathbb{N},\, i \geq 1,
      \\
      &B(n-1)\dots[B(s_{i}) + Q(s_{i})]v(s_{i})\text{ if }n\in \left(s_{i},s_{i+1}\right) \text{ for certain }i\in \mathbb{N},\, i \geq 1.
   \end{cases}
\end{equation*}
A computation shows that $\left(v(n)\right)_{n\in \mathbb{N}}$ is the solution of system
\begin{equation*}
v(n+1)=(B(n)+Q(n))v(n)
\end{equation*}
with initial condition
\begin{equation}
   v_{0} := B^{-1}\left( 0\right) \dots B^{-1}\left(s_{1}-1\right) v(s_1) \label{sol}.
\end{equation}%
Observe that $Q(n)=0\text{ for }n\neq s_{i},\text{ }i\in \mathbb{N}$ and that from (ii) of the properties of $Q_\ell^{(j)}$, $\ell\in\mathbb N$, $j\in\{1,\dots,\ell\}$, it follows that
\begin{equation*}
   \Vert Q(s_{i})\Vert \leq \varepsilon_{i}b\text{ for }i\in \mathbb{N}
\end{equation*}
and in particular, by definition of $\varepsilon_i$, we have $B + Q \in \mathcal L^{\mathrm{Lya}}(\mathbb N,\mathbb R^{k \times k})$ and
\begin{equation*}
   \underset{l\rightarrow \infty }{\lim }Q(l)=0.
\end{equation*}

We will show that for the solution $\left( v(n,v_{0})\right) _{n\in \mathbb{N}}$ the inequality
\begin{equation}
   \Vert v(\tau _{q},v_{0})\Vert \geq \frac{1}{2}\sin \varepsilon_{q}\;\Vert\Phi_{B}(\tau_{q},s_{q})\Vert \Vert v(s_{q},v_{0})\Vert,
   \label{Perturbed_Bohl_blr13}
\end{equation}
is satisfied for all $q\in \mathbb{N}$, $q\geq 1$.
Let us fix $q_{0}\in\mathbb{N}$, $q_{0}\geq 1$.
We have
\begin{equation}
   v(s_{q_{0}},v_{0}) = \underset{j\rightarrow \infty }{\lim}z_{\ell_{j}^{(q_{0})}}\left( s_{q_{0}},z_{\ell_{j}^{(q_{0})},0}\right).
   \label{Perturbed_Bohl_blr14}
\end{equation}%
Since $s_{q-1}<\tau _{q}<s_{q},$ $Q(i)=0$ for $i=\tau _{q},\dots,s_{q}-1$
and
\begin{align}
   v(\tau _{q_{0}},v_{0}) &= B^{-1}\left( \tau _{q_{0}}\right) \dots B^{-1}\left(s_{q_{0}}-1\right) v(s_{q_{0}},v_{0})
   \nonumber
   \\
   &= B^{-1}\left( \tau _{q_{0}}\right) \dots B^{-1}\left( s_{q_{0}}-1\right)\underset{j\rightarrow \infty }{\lim }z_{\ell_{j}^{(q_{0})}}\left(s_{q_{0}},z_{\ell_{j}^{(q_{0})},0}\right)
   \nonumber
   \\
   &= \underset{j\rightarrow \infty }{\lim }B^{-1}\left( \tau _{q_{0}}\right)\dots B^{-1}\left( s_{q_{0}}-1\right) z_{\ell_{j}^{(q_{0})}}\left(s_{q_{0}},z_{\ell_{j}^{(q_{0})},0}\right)
   \nonumber
   \\
   &= \underset{j\rightarrow \infty }{\lim}z_{\ell_{j}^{(q_{0})}}\left(\tau_{q_{0}},z_{\ell_{j}^{(q_{0})},0}\right).
   \label{Perturbed_Bohl_blr15}
\end{align}
By (\ref{Perturbed_Bohl_blr11}) we have
\begin{equation*}
   \Vert z_{\ell_{j}^{(q_{0})}}(\tau _{q_{0}},z_{\ell_{j}^{(q_{0})},0})\Vert
   \geq
   \frac{1}{2}\sin \varepsilon _{q_{0}}\;\Vert \Phi _{B}(\tau_{q_{0}},s_{q_{0}})\Vert \Vert z_{\ell_{j}^{(q_{0})}}(s_{q_{0}},z_{\ell_{j}^{(q_{0})},0})\Vert,
\end{equation*}%
for all $j\in \mathbb{N}$ such that $q_{0} \leq \ell_{j}^{(q_{0})}$.
Passing to the limit when $j$ tends to infinity in the last inequality and having in mind \eqref{Perturbed_Bohl_blr14} and \eqref{Perturbed_Bohl_blr15} we get \eqref{Perturbed_Bohl_blr13}.

Finally we will show that
\begin{equation}
   \underline{\beta }_{B+Q}(v_{0})<0,
   \label{Perturbed_Bohl_blr16}
\end{equation}
which proves (iii).
From \eqref{Perturbed_Bohl_blr13} we have
\begin{equation*}
   \frac{\Vert v(s_{q},v_{0})\Vert }{\Vert v(\tau _{q},v_{0})\Vert }\leq \frac{2}{\sin \varepsilon _{q}}\Vert \Phi _{B}^{-1}(s_{q},\tau _{q},)\Vert^{-1}
\end{equation*}
and using \eqref{_1} and \eqref{_2} we get
\begin{equation*}
   \frac{1}{s_{q}-\tau _{q}}\ln \frac{\Vert v(s_{q},v_{0})\Vert }{\Vert v(\tau_{q},v_{0})\Vert } \leq -\delta + 2\varepsilon_q,
   \qquad q \in \mathbb N.
\end{equation*}
The last inequality implies (\ref{Perturbed_Bohl_blr16}).
\end{proof}

Using Remark \ref{Rem:eps} we obtain from Lemma \ref{Lem:10} the following
result.
\begin{corollary}[Perturbation with special solution]
\label{C2}Under the assumption of Lemma \ref{Lem:10} for each $\varepsilon >0
$ there exists a $Q_{\varepsilon }\in \mathcal{L}^{\infty }(\mathbb{N},%
\mathbb{R}^{k\times k})$ such that $\left\Vert Q_{\varepsilon }\right\Vert
_{\infty }\leq \varepsilon $, $B+Q_{\varepsilon }\in \mathcal{L}^{\mathrm{Lya%
}}(\mathbb{N},\mathbb{R}^{k\times k})$ and $\underset{x_{0}\in \mathbb{R}%
^{k}\setminus \{0\}}{\inf }\underline{\beta }_{B+Q}(x_{0})<0$.
\end{corollary}

\section{Upper triangularization and subsystems}

In this section we define an equivalence relation between systems which preserves Bohl exponents and we show that each equivalence class contains an upper triangular system. This tool is important in the proof of our main result to repeatedly construct perturbations for upper triangular systems which then carry over their dynamic properties to equivalent systems.

\begin{definition}[Dynamic equivalence]\label{def:dynamic-equivalence}
Let $A, B \in \mathcal{L}^{\mathrm{Lya}}(\mathbb{N},\mathbb{R}^{d\times d})$. The two systems
\begin{equation}\label{two-systems}
      x(n+1) = A(n) x(n)
      \quad \text{and} \quad
      y(n+1) = B(n) y(n),
      \qquad
      n \in \mathbb{N},
\end{equation}
are called \emph{dynamically equivalent} (or \emph{kinematically similar}), if there exists $T \in \mathcal{L}^{\mathrm{Lya}}(\mathbb{N},\mathbb{R}^{d\times d})$ with
\begin{equation*}
   B(n)
   =
   T(n+1)^{-1} A(n) T(n),
   \quad
   n \in \mathbb{N}.
\end{equation*}
$T$ is called \emph{Lyapunov transformation} between the two systems \eqref{two-systems}. The two systems \eqref{two-systems} are said to be \emph{dynamically equivalent} (or \emph{kinematically similar) via $T$}.
\end{definition}

To prepare the construction of an upper triangular system consider a
$k$-dimensional subspace $L\neq \left\{ 0\right\} $ of $\mathbb{R}^{d}$ and
let $\left\{ l_{1},\dots,l_{k}\right\} $ be a base of $L.$ It is well-known that there are vectors $l_{k+1},\dots,l_{d}\in $ $\mathbb{R}^{d}$ such
that $\left\{ l_{1},\dots,l_{d}\right\} $ is a base of $\mathbb{R}^{d}.$ We
define
\begin{equation*}
l_{i}(n)=\Phi _{A}\left( n,0\right) l_{i},\quad n\in \mathbb{N},
i=1,\dots,d.
\end{equation*}%
Since the $A(n)$ are invertible,  the vectors $\left\{
l_{1}(n),\dots,l_{d}(n)\right\} $ form a base of $\mathbb{R}^{d}$ for each $%
n\in \mathbb{N}$. Now we will orthonormalize the base $\left\{
l_{1}(n),\dots,l_{d}(n)\right\} $ for each $n\in \mathbb{N}$ using the
Gram--Schmidt process. We define vectors $\overline{l}_{1}(n),\dots,%
\overline{l}_{d}(n)\in $ $\mathbb{R}^{d}$ as follows
\begin{equation*}
\widetilde{l}_{1}(n)=l_{1}(n), \quad \overline{l}_{1}(n)=\frac{\widetilde{l}%
_{1}(n)}{\big\Vert \widetilde{l}_{1}(n)\big\Vert },
\end{equation*}%
\begin{equation*}
\widetilde{l}_{i}(n)=l_{i}(n)-\sum_{j=1}^{i-1}\left\langle l_{i}(n),%
\overline{l}_{j}(n)\right\rangle \overline{l}_{j}(n),\quad \overline{l}%
_{i}(n)=\frac{\widetilde{l}_{i}(n)}{\big\Vert \widetilde{l}%
_{i}(n)\big\Vert }, i=2,\dots,d.
\end{equation*}%
Then we have
\begin{equation}
l_{1}(n)=\overline{l}_{1}(n)\big\Vert \widetilde{l}_{1}(n)\big\Vert ,
\label{GS2}
\end{equation}%
\begin{equation}
l_{i}(n)=\overline{l}_{i}(n)\big\Vert \widetilde{l}_{i}(n)\big\Vert
+\sum_{j=1}^{i-1}\left\langle l_{i}(n),\overline{l}_{j}(n)\right\rangle
\overline{l}_{j}(n).  \label{GS3}
\end{equation}%
It is well-known that the vectors $\left\{ \overline{l}_{1}(n),\dots,%
\overline{l}_{d}(n)\right\}$ form an orthonormal base of $\mathbb{R}^{d}.
$ If we define $V(n)$ and $U(n)$ to be the matrices whose columns are
\begin{equation*}
l_{1}(n),\dots,l_{d}(n)
\quad\text{and}\quad
\overline{l}_{1}(n),\dots,\overline{l}_{d}(n),
\end{equation*}%
respectively,
then (\ref{GS2})-(\ref{GS3}) may be rewritten in the following form%
\begin{equation*}
V(n)=U(n)C(n),
\end{equation*}%
where
\begin{equation*}
C(n)=\left[
\begin{array}{cccc}
\big\Vert \widetilde{l}_{1}(n)\big\Vert  & \left\langle l_{2}(n),%
\overline{l}_{1}(n)\right\rangle  & \dots & \left\langle l_{d}(n),\overline{l}%
_{1}(n)\right\rangle  \\
0 & \big\Vert \widetilde{l}_{2}(n)\big\Vert  & \dots & \left\langle
l_{d}(n),\overline{l}_{2}(n)\right\rangle  \\
\vdots  &  & \ddots & \vdots  \\
0 & \dots & \dots & \big\Vert \widetilde{l}_{d}(n)\big\Vert
\end{array}%
\right].
\end{equation*}%
In particular we have
\begin{equation}
\operatorname{span}\left\{ l_{1}(n),\dots,l_{i}(n)\right\} =\operatorname{span}\left\{ \overline{l}%
_{1}(n),\dots,\overline{l}_{i}(n)\right\} \quad, i=1,\dots,d.
\label{SG3.5}
\end{equation}%
By the definition of $U(n)$ and $C(n)$ it is clear that $U(n)$ is
orthonormal and $C(n)$ is upper triangular. Consider the sequence $B=\left(
B(n)\right) _{n\in \mathbb{N}}$ with
\begin{equation}
B(n)\coloneqq U^{T}(n+1)A(n)U(n)=U^{-1}(n+1)A(n)U(n).  \label{GS3.6}
\end{equation}%
Since $A\in \mathcal{L}^{\mathrm{Lya}}(\mathbb{N},\mathbb{R}^{d\times d})$
and $U(n)$ is orthonormal, $B\in \mathcal{L}^{\mathrm{Lya}}(\mathbb{N},%
\mathbb{R}^{d\times d}).$ Moreover
\begin{align*}
C(n+1)&=U^{-1}(n+1)V(n+1)=U^{-1}(n+1)A(n)V(n)
\\
&= U^{-1}(n+1)A(n)U(n)C(n)=B(n)C(n)
\end{align*}%
and therefore
\begin{equation*}
B(n)=C(n+1)C^{-1}(n).
\end{equation*}%
Consequently $B(n)$ is upper triangular due to the fact that $C(n)$ is upper triangular.

The above considerations show that system (\ref{1}) is dynamically
equivalent to the upper triangular system
\begin{equation}
   y(n+1)=B(n)y(n)
   \label{GS4}
\end{equation}
and the transformation
\begin{equation*}
   x(n)=U(n)y(n)
\end{equation*}
establishes this equivalence and also maps the subspace $L \subseteq \mathbb{R}^d$ of \eqref{1} to the  subspace $\mathbb{R}^k \times \{0\} \subseteq \mathbb{R}^d$ of system \eqref{GS4}, see also Lemma \ref{L-systemInvariance}.
Denote $B(n)=\left[ b_{ij}(n)\right] _{i,j=1,..,d}$, define
\begin{equation*}
   B_{1}(n) \coloneqq \left[ b_{ij}(n)\right] _{i,j=1,..,k}
\end{equation*}
and consider the system
\begin{equation}
   y_1(n+1)=B_{1}(n)y_1(n).
   \label{GS4.5}
\end{equation}

\begin{definition}
System (\ref{GS4.5}) with coefficient matrix $A_L \coloneqq B_1 \in \mathcal{L}^{\mathrm{Lya}}(\mathbb{N}, \mathbb{R}^{k \times k})$ is called $L$-subsystem of system (\ref{1}).
\end{definition}

The following two lemmas show that the sequence $(L, \Phi_A(1,0)L, \Phi_A(2,0)L, \dots)$ of subspaces of $\mathbb{R}^d$ is mapped by the dynamic equivalence to the constant sequence $(\mathbb{R}^k \times \{0\},\mathbb{R}^k \times \{0\},\mathbb{R}^k \times \{0\},\dots)$ which renders $\mathbb{R}^k \times \{0\}$ as an invariant subspace for the upper triangular system \eqref{GS4}. The restriction of \eqref{GS4} to this invariant subspace yields the $L$-subsystem \eqref{GS4.5} of \eqref{1}.

\begin{lemma}[$L$ becomes invariant under dynamic equivalence]
\label{L-systemInvariance}
We have
\begin{equation}
U^{-1}(n)\Phi _{A}\left( n,0\right) L=\operatorname{span}\left\{ e_{1},\dots,e_{k}\right\} ,
\label{GS6}
\end{equation}%
and%
\begin{equation}
\Phi _{B}\left( n,0\right) U^{-1}(0)L=\operatorname{span}\left\{ e_{1},\dots,e_{k}\right\}
\label{GS7}
\end{equation}%
for each $n\in \mathbb{N}$.
\end{lemma}

\begin{proof}
\bigskip Since $\left\{ l_{1},\dots,l_{k}\right\} $ is a base of $L$,
\begin{equation*}
L=\operatorname{span}\left\{ l_{1},\dots,l_{k}\right\}
\end{equation*}%
and
\begin{align*}
\Phi _{A}\left( n,0\right) L&=\operatorname{span}\left\{ \Phi _{A}\left( n,0\right)
l_{1},\dots,\Phi _{A}\left( n,0\right) l_{k}\right\}
\\
&=
\operatorname{span}\{ l_{1}(n),\dots,l_{k}(n)\}
\\
&=
\operatorname{span}\{ \overline{l}_{1}(n),\dots,\overline{l}_{k}(n)\}.
\end{align*}%
The last equality follows from (\ref{SG3.5}) with $i=k$. Finally, using the
fact that $U^{-1}(n)\overline{l}_{i}(n)=e_{i}$, $i=1,\dots,d,$ $n\in \mathbb{N}
$, we get
\begin{align*}
U^{-1}(n)\Phi _{A}\left( n,0\right) L&=\operatorname{span}\left\{ U^{-1}(n)\overline{l}%
_{1}(n),\dots,U^{-1}(n)\overline{l}_{k}(n)\right\}
\\
&=
\operatorname{span}\left\{ e_{1},\dots,e_{k}\right\} .
\end{align*}%
This proves (\ref{GS6}). Equality (\ref{GS7}) follows from (\ref{GS6}) and
the relation%
\begin{equation*}
U^{-1}(n)\Phi _{A}\left( n,0\right) U(0)=\Phi _{B}\left( n,0\right) .
\end{equation*}
\end{proof}

\begin{lemma}[$L$ becomes invariant under dynamic equivalence II]
\label{GS9}Each solution $\left( y\left( n,y_{0}\right) \right) _{n\in
\mathbb{N}}$ of system (\ref{GS4}) with $y_{0}\in \operatorname{span}\left\{
e_{1},\dots,e_{k}\right\} $ has the form%
\begin{equation}
y\left( n,y_{0}\right) =\left[
\begin{array}{c}
y_{1}\left( n,y_{0,1}\right)  \\
0
\end{array}%
\right] ,  \label{GS8}
\end{equation}%
where $\left( y_{1}\left( n,y_{0,1}\right) \right) _{n\in \mathbb{N}}$
is a solution of (\ref{GS4.5}) and $0$ is the zero vector of $\mathbb{R%
}^{d-k}$ and vice versa, if $\left( y_{1}\left( n,y_{0,1}\right) \right)
_{n\in \mathbb{N}}$ is a solution of (\ref{GS4.5}), then the formula (\ref%
{GS8}) gives the solution of (\ref{GS4}) with initial condition
\begin{equation}
y_{0}=\left[
\begin{array}{c}
y_{0,1} \\
0
\end{array}%
\right] .  \label{GS10}
\end{equation}
\end{lemma}

\begin{proof}
Suppose that $y_{0}\in \operatorname{span}\left\{ e_{1},\dots,e_{k}\right\} $, then $\left(
y\left( n,y_{0}\right) \right) _{n\in \mathbb{N}}$ is a solution of system
(\ref{GS4}). It is clear that $y_{0}$ has the following form
\begin{equation*}
\left[
\begin{array}{c}
y_{0,1} \\
0
\end{array}%
\right] .
\end{equation*}%
with $y_{0,1}\in \mathbb{R}^{k}.$ Let us denote
\begin{equation*}
B(n)=\left[
\begin{array}{cc}
B_{1}(n) & B_{12}(n) \\
0 & B_{2}(n)%
\end{array}%
\right] ,
\end{equation*}%
then $\Phi _{B}\left( n,m\right) $ has the following form
\begin{equation*}
\Phi _{B}\left( n,m\right) =\left[
\begin{array}{cc}
\Phi _{B_{1}}\left( n,m\right)  & \Psi _{12}(n,m) \\
0 & \Psi _{2}(n,m)%
\end{array}%
\right]
\end{equation*}%
and therefore
\begin{align*}
y\left( n,y_{0}\right) &=\left[
\begin{array}{cc}
\Phi _{B_{1}}\left( n,0\right)  & \Psi _{12}(n,0) \\
0 & \Psi _{2}(n,0)%
\end{array}%
\right] \left[
\begin{array}{c}
y_{0,1} \\
0
\end{array}%
\right]
\\
&=
\left[
\begin{array}{c}
\Phi _{B_{1}}\left( n,0\right) y_{0,1} \\
0
\end{array}%
\right] =\left[
\begin{array}{c}
y_{1}\left( n,y_{0,1}\right)  \\
0
\end{array}%
\right] .
\end{align*}%
Conversely, if $\left( y_{1}\left( n,y_{0,1}\right) \right) _{n\in
\mathbb{N}}$ is a solution of (\ref{GS4.5}) and we define
\begin{equation*}
y( n) =\left[
\begin{array}{c}
y_{1}\left( n,y_{0,1}\right)  \\
0
\end{array}%
\right], \quad n\in \mathbb{N}\text{,}
\end{equation*}%
then $y_{0}:=y(0)\in \operatorname{span}\left\{ e_{1},\dots,e_{k}\right\} $ and
\begin{align*}
y(n+1)&=\left[
\begin{array}{c}
y_{1}\left( n+1,y_{0,1}\right)  \\
0
\end{array}%
\right]
\\ &=
\left[
\begin{array}{c}
B_{1}(n)y_{1}\left( n,y_{0,1}\right)  \\
0
\end{array}%
\right] =\left[
\begin{array}{cc}
B_{1}(n) & B_{12}(n) \\
0 & B_{2}(n)%
\end{array}%
\right] \left[
\begin{array}{c}
y_{1}\left( n,y_{0,1}\right)  \\
0
\end{array}%
\right]
\\ &=
B(n)y(n).
\end{align*}%
\end{proof}

The following lemma states that the upper triangular normal form of system \eqref{1}, as well as an associated $L$-subsystem have Bohl exponents which are preserved under the dynamic equivalence and projection, respectively.

\begin{lemma}[Bohl exponents of $L$-subsystem]
\label{GS11}
We have
\begin{equation*}
   \overline{\beta }_{A_L}(\mathbb{R}^{\dim L}) = \overline{\beta }_{A}(L),
   \qquad
   \underline{\beta }_{A_L}(\mathbb{R}^{\dim L})=\underline{\beta }_{A}(L)
\end{equation*}
and
\begin{equation*}
   \overline{\beta }_{A_L}(y_{0,1})
   = \overline{\beta }_{B}(y_{0})
   = \overline{\beta }_{A}(U(0)y_{0}),
   \qquad
   \underline{\beta }_{A_L}(y_{0,1})=\underline{\beta }_{B}(y_{0})
   = \underline{\beta }_{A}(U(0)y_{0})
\end{equation*}
for any $y_{0}\in \operatorname{span}\left\{ e_{1},\dots,e_{k}\right\} $ and $y_{0,1}$ given
by (\ref{GS10}).
\end{lemma}

\begin{proof}
By Lemma 31 in \cite{CzornikEtal2023} we have
\begin{equation*}
\overline{\beta }_{A}(L)=\overline{\beta }_{B}(U^{-1}(0)L),\qquad
\underline{\beta }_{B_{1}}(\mathbb{R}^{k})=\underline{\beta }_{A}(L)
\end{equation*}%
and%
\begin{equation*}
\overline{\beta }_{B}(y_{0})=\overline{\beta }_{A}(U(0)y_{0}),\qquad%
\underline{\beta }_{B}(y_{0})=\underline{\beta }_{A}(U(0)y_{0}).
\end{equation*}%
Equality (\ref{GS7}) implies that
\begin{equation*}
U^{-1}(0)L=\operatorname{span}\left\{ e_{1},\dots,e_{k}\right\} .
\end{equation*}%
Using Lemma \ref{GS9} we conclude that
\begin{equation*}
\left\Vert y\left( n,y_{0}\right) \right\Vert = \Vert y_{1}(n,y_{0,1}) \Vert ,
\end{equation*}%
where $\left( y\left( n,y_{0}\right) \right) _{n\in \mathbb{N}}$ and $%
\left( y_{1}\left( n,y_{0,1}\right) \right) _{n\in \mathbb{N}}$ are solutions of (\ref{GS4}) and (\ref{GS4.5}), respectively, $y_{0}\in
\operatorname{span}\left\{ e_{1},\dots,e_{k}\right\} $ and $y_{0,1}$ is given by (\ref{GS10}%
). Now the conclusion of the lemma follows directly from the definition of
upper and lower Bohl exponents.
\end{proof}

\begin{lemma}[Dichotomies of $L$-subsystems]
\label{Lem:GS}
Let $A \in \mathrm{BD}^{d}$ and $L_{1} \oplus L_{2} = \mathbb{R}^d$ be the corresponding Bohl dichotomy splitting. Then

(i) $A_{L_{1}} \in \mathrm{BD}^{\dim L_{1}}$ and $A_{L_{2}} \in \mathrm{BD}^{\dim L_{2}}$,

(ii) if $A \notin \mathrm{ED}^{d}$ then $A_{L_{1}} \notin \mathrm{ED}^{\dim L_{1}}$ or $A_{L_{2}} \notin \mathrm{ED}^{\dim L_{2}}$.
\end{lemma}

\begin{proof}
From Lemma \ref{lem:BD-characterization}(ii) we know that
\begin{equation*}
   \sup_{x_{0}\in L_{1}\setminus \{0\}}\overline{\beta }_{A}(x_{0})<0
   \qquad\text{and}\qquad
   \inf_{x_{0}\in L_{2}\setminus \{0\}}\underline{\beta }_{A}(x_{0})>0.
\end{equation*}
The last two inequalities together with Lemma \ref{GS11} yield
\begin{equation*}
   \sup_{x_{0} \in \mathbb{R}^{\dim L_{1}}\setminus \{0\}}\overline{\beta }_{A_{L_{1}}}(x_{0}) < 0
   \qquad\text{and}\qquad
   \inf_{x_{0} \in \mathbb{R}^{\dim L_{2}}\setminus \{0\}}\underline{\beta }_{A_{L_{2}}}(x_{0}) > 0.
\end{equation*}
Using again Lemma \ref{lem:BD-characterization}(ii) we conclude that $A_{L_{1}} \in \mathrm{BD}^{\dim L_{1}}$ and $A_{L_{2}} \in \mathrm{BD}^{\dim L_{2}}$.

Suppose that (ii) does not hold, i.e.
\begin{equation*}
   A_{L_{1}} \in \mathrm{ED}^{\dim L_{1}}
   \qquad\text{and}\qquad
   A_{L_{2}} \in \mathrm{ED}^{\dim L_{2}}.
\end{equation*}
This implies by Lemma \ref{sup_ED} that
\begin{equation*}
   \overline{\beta }_{A_{L_{1}}}(\mathbb{R}^{\dim L_{1}}) < 0
   \qquad \text{and} \qquad
   \underline{\beta }_{A_{L_{2}}}(\mathbb{R}^{\dim L_{2}}) > 0
\end{equation*}
and by Lemma \ref{GS11} we have
\begin{equation*}
\overline{\beta }_{A}(L_{1})<0
\qquad \text{and} \qquad
\underline{\beta }_{A}(L_{2})>0.
\end{equation*}%
The last inequalities, together with Lemma \ref%
{lem:ED-characterization}(ii) means that $A\in \mathrm{ED}^{d}.$ The obtained
contradiction completes the proof.
\end{proof}

\begin{lemma}[Lifting a perturbation from $A_L$ to $A$, I]
\label{lem:subsystem_0}
Let $A\in \mathcal{L}^{\mathrm{Lya}}(\mathbb{N},%
\mathbb{R}^{d\times d})$ and $L$ is a $k$-dimensional subspace of $\mathbb{R}%
^{d}$ such that there exist $Q_{1}\in \mathcal{L}^{\infty }(\mathbb{N},%
\mathbb{R}^{k\times k})$ and $x_{0,1}\in \mathbb{R}^{k}\backslash \{0\}$ with

(i) $\underset{l\rightarrow \infty }{\lim }Q_{1}(l)=0$,

(ii) $A_{L}+Q_{1}\in \mathcal{L}^{\mathrm{Lya}}(\mathbb{N}%
,\mathbb{R}^{k\times k})$,

(iii) $\underline{\beta }_{A_{L} + Q_1}(x_{0,1})\leq 0\leq
\overline{\beta }_{A_{L} + Q_1}(x_{0,1}).$

Then there exist $Q\in \mathcal{L}^{\infty }(\mathbb{N},\mathbb{R}^{d\times
d})$ and $x_{0}\in \mathbb{R}^{d}\backslash \{0\}$ with

(i) $\underset{l\rightarrow \infty }{\lim }Q(l)=0$,

(ii) $A+Q\in \mathcal{L}^{\mathrm{Lya}}(\mathbb{N},\mathbb{R}^{d\times d})$,

(iii) $\underline{\beta }_{A+Q}(x_{0})\leq 0\leq \overline{\beta }_{A+Q}(x_{0}).$

Moreover, if one of the inequalities in assumption (iii) is strict, then the
appropriate inequality in thesis (iii) is also strict.
\end{lemma}

\begin{proof}
Consider system (\ref{GS4.5}).
We know that it is upper triangular and that $U$ establishes the dynamical equivalence between $A$ and $B$.
Let us perturb system (\ref{GS4.5}) by the perturbation $Q_{2}=\left( Q_{2}\left( n\right)\right) _{n\in \mathbb{N}}$, where
\begin{equation*}
Q_{2}\left( n\right) =\left[
\begin{array}{cc}
Q_{1}(n) & 0 \\
0 & 0
\end{array}
\right] .
\end{equation*}
It is clear that
\begin{equation*}
\lim_{\ell \rightarrow \infty }Q_{2}(\ell )=0.
\end{equation*}
By Lemma \ref{GS9} we know that $\left( y\left( n,y_{0}\right) \right)
_{n\in \mathbb{N}},$ where%
\begin{equation*}
y_{0}=\left[
\begin{array}{c}
x_{0,1} \\
0
\end{array}%
\right] \text{ and }y\left( n,y_{0}\right) =\left[
\begin{array}{c}
y_{1}\left( n,x_{0,1}\right)  \\
0
\end{array}%
\right] ,
\end{equation*}%
is a solution of system (\ref{GS4.5}).
From Lemma \ref{GS11} we get by the assumption (iii) that
\begin{equation}
\underline{\beta }_{B+Q_{2}}(y_{0})\leq 0\text{ and }\overline{\beta }%
_{B+Q_{2}}(y_{0})\geq 0.  \label{GS16}
\end{equation}
The Lyapunov transformation $U^{-1}$ establishes the dynamical equivalence
between $B+Q_{2}$ and $A+Q,$ where $Q=\left( Q(n)\right) _{n\in \mathbb{N}}$
and
\begin{equation}
Q(n)=U\left( n+1\right) Q_{2}(n)U^{-1}(n).  \label{GS22}
\end{equation}%
Observe that $\lim_{\ell \rightarrow \infty }Q(\ell )=0.$ Moreover by Lemma
31 in \cite{CzornikEtal2023} we have
\begin{equation*}
\underline{\beta }_{A+Q}(U(0)y_{0})=\underline{\beta }_{B+Q_{2}}(y_{0})\leq 0
\end{equation*}%
and%
\begin{equation*}
\overline{\beta }_{A+Q}(U(0)y_{0})=\overline{\beta }_{B+Q_{2}}(y_{0})\geq 0.
\end{equation*}%
Therefore the statement of the lemma is true with $Q$ defined by (\ref{GS22}%
) and $x_{0}=U(0)y_{0}.$ The proof of the case when one of the inequalities
in assumption (iii) is strict is analogical.
\end{proof}

The proof of the next Lemma is analogous to the proof of Lemma \ref{lem:subsystem_0}.

\begin{lemma}[Lifting a perturbation from $A_L$ to $A$, II]
\label{lem:subsystem_eps}
Let $A\in \mathcal{L}^{\mathrm{Lya}}(\mathbb{N},%
\mathbb{R}^{d\times d})$ and $L$ is a $k$-dimensional subspace of $\mathbb{R}%
^{d}$ such that for any $\varepsilon >0$ there exist $Q_{1}\in \mathcal{L}%
^{\infty }(\mathbb{N},\mathbb{R}^{k\times k})$ and $x_{0,1}\in \mathbb{R}%
^{k}\backslash \{0\}$ with

(i) $\left\Vert Q_{1}\right\Vert _{\infty }<\varepsilon $,

(ii) $A_{L} +Q_{1}\in \mathcal{L}^{\mathrm{Lya}}(\mathbb{N}%
,\mathbb{R}^{k\times k})$,

(iii) $\underline{\beta }_{A_{L}+Q_{1}}(x_{0,1})\leq 0\leq
\overline{\beta }_{A_{L}+Q_{1}}(x_{0,1}).$

Then for any $\varepsilon >0$ there exist $Q\in \mathcal{L}^{\infty }(%
\mathbb{N},\mathbb{R}^{d\times d})$ and $x_{0}\in \mathbb{R}^{d}\backslash
\{0\}$ with

(i) $\left\Vert Q\right\Vert _{\infty }<\varepsilon $,

(ii) $A+Q\in \mathcal{L}^{\mathrm{Lya}}(\mathbb{N},\mathbb{R}^{d\times d})$,

(iii) $\underline{\beta }_{A+Q}(x_{0})\leq 0\leq \overline{\beta }_{A+Q}(x_{0}).$

Moreover, if one of the inequalities in assumption (iii) is strict, then the
appropriate inequality in thesis (iii) is also strict.
\end{lemma}

\begin{lemma}[Perturbation with no Bohl dichotomy]
\label{lem:lya_perturbed_lower_v2_step3}Let $B\in \mathcal{L}^{\mathrm{Lya}}(%
\mathbb{N},\mathbb{R}^{k\times k})$ be such that

(i) $B\in \mathrm{BD}^{k}\mathrm{\backslash ED}^{k}$

(ii)$\underset{x_{0}\in \mathbb{R}^{k}\setminus \{0\}}{\inf }\overline{\beta
}_{B}(x_{0})<0$

(iii) $\underset{x_{0}\in \mathbb{R}^{k}\setminus \{0\}}{\sup }\underline{%
\beta }_{B}(x_{0})>0.$

Then for any $\varepsilon >0$ there exists a $Q\in \mathcal{L}^{\infty }(%
\mathbb{N},\mathbb{R}^{k\times k})$ with

(i) $\left\Vert Q\right\Vert _{\infty }<\varepsilon $,

(ii) $B+Q\in \mathcal{L}^{\mathrm{Lya}}(\mathbb{N},\mathbb{R}^{k\times k})$,

(iii) $B+Q\notin \mathrm{BD}^{k}\mathrm{.}$
\end{lemma}

\begin{proof}
Let us fix $\varepsilon >0$ and let $L_{1} \oplus L_{2}$ be the decomposition of $\mathbb{R}^{k}$ from the definition of Bohl dichotomy for system $B$.
By the assumptions (ii) and (iii), the dimension of $L_1$ and of $L_2$ is greater or equal then 1.

To prove the lemma, we use a recursive argument and define $B^{(0)} := B$ and $L_1^{(0)} := L_1$, $L_2^{(0)} := L_2$.
Then by Lemma \ref{Lem:GS}, $B_{L_1^{(0)}}^{(0)}$ and $B_{L_2^{(0)}}^{(0)}$ have a Bohl dichotomy, and $B_{L_1^{(0)}}^{(0)}$ or $B_{L_2^{(0)}}^{(0)}$ has no exponential dichotomy.
Now for $\mu \in \mathbb N$ suppose that $B^{(\mu)}$ has a Bohl dichotomy with splitting $L_1^{(\mu)} \oplus L_2^{(\mu)}$ and at least one of the following two cases holds:

(i) $B_{L_1^{(\mu)}}^{(\mu)}$ has no exponential dichotomy or

(ii) $B_{L_2^{(\mu)}}^{(\mu)}$ has no exponential dichotomy.

Note that by Lemma \ref{Lem:GS}(i), the fact that $B^{(\mu)}$ has a Bohl dichotomy with splitting $L_1^{(\mu)} \oplus L_2^{(\mu)}$ implies that $B_{L_1^{(\mu)}}^{(\mu)}$ and $B_{L_2^{(\mu)}}^{(\mu)}$ also have a Bohl dichotomy.

In case (i) $r := \dim L_1^{(\mu)} \geq 2$, since the notions of Bohl and exponential dichotomy on one-dimensional subspaces coincide, $\underset{x_{0} \in \mathbb{R}^{r}\backslash \{0\}}{\sup }\overline{\beta }_{B_{L_{1}^{(\mu)}}^{(\mu)}}(x_{0}) < 0$ by Lemma \ref{lem:BD-characterization} and Lemma \ref{GS11} and $\overline{\beta}_{B_{L_{1}^{(\mu)}}^{(\mu)}}(\mathbb{R}^{r})\geq 0$ by Lemma \ref{lem:ED-characterization}.
Hence we can apply Corollary \ref{C1} to $B_{L_{1}}^{(\mu)}$ and any $x_{0,1} \in \mathbb R^r \setminus \{0\}$ and construct $Q_{\varepsilon }\in  \mathcal{L}^{\infty }(\mathbb{N},\mathbb{R}^{r\times r})$ such that $\left\Vert Q_{\varepsilon }\right\Vert _{\infty}\leq \varepsilon $, $B_{L_{1}^{(\mu)}}^{(\mu)}+Q_{\varepsilon }\in
\mathcal{L}^{\mathrm{Lya}}(\mathbb{N},\mathbb{R}^{r\times r})$ and
\begin{equation*}
   \underline{\beta }_{B_{L_{1}}^{(\mu)}+Q_{\varepsilon}}(x_{0,1}) \leq 0
   \qquad\text{and}\qquad
   \overline{\beta }_{B_{L_{1}}^{(\mu)}+Q_{\varepsilon }}(x_{0,1})\geq 0.
\end{equation*}
Next using Lemma \ref{lem:subsystem_eps} $\mu+1$ times we can construct $Q\in \mathcal{L}^{\infty }(\mathbb{N},\mathbb{R}^{k\times k})$ and $x_{0}\in\mathbb{R}^{k}$ such that $\left\Vert Q\right\Vert _{\infty }\leq \varepsilon $, $B+Q \in \mathcal{L}^{\mathrm{Lya}}(\mathbb{N},\mathbb{R}^{k \times k})$ and
\begin{equation*}
   \underline{\beta }_{B+Q}(x_{0})\leq 0
   \qquad\text{and}\qquad
   \overline{\beta }_{B+Q}(x_{0})\geq 0.
\end{equation*}
The last inequalities imply, in the light of Corollary \ref{cor:notBD}, that $B+Q\notin \mathrm{BD}^{k}$, which concludes the proof.

In case (ii), we set $B^{(\mu+1)} := B_{L_2^{(\mu)}}^{(\mu)}$ and continue with the recursion.
However, the recursion terminates eventually with case (i) since $\operatorname{dim} L_2^{(0)} > \operatorname{dim} L_2^{(1)} > \ldots > 1$, whereby the case $\operatorname{dim} L_2^{(\mu)} = 1$ cannot occur, since on one-dimensional spaces the notion of Bohl and exponential dichotomy coincide and $B^{(\mu)}$ has a Bohl dichotomy.
\end{proof}

\begin{lemma}[Decaying perturbation and exponential dichotomy]
\label{lem:ED_and_lim0}
Let $B\in \mathcal{L}^{\mathrm{Lya}}(\mathbb{N},\mathbb{R}^{k\times k})$ and $Q \in \mathcal{L}^\infty(\mathbb{N},\mathbb{R}^{k\times k})$, such that $B + Q \in \mathcal{L}^{\mathrm{Lya}}(\mathbb{N},\mathbb{R}^{k\times k})$ and $\lim\limits_{n\to\infty}Q(n) = 0$.
If $B \notin \mathrm{ED}^k$, then $B + Q \notin \mathrm{ED}^k$.
\end{lemma}

\begin{proof}
Suppose by contradiction that $B + Q \in \mathrm{ED}^k$.
The set $\mathrm{ED}^{k}$ is open in $\left( \mathcal{L}^{\mathrm{Lya}}(\mathbb{N},\mathbb{R}^{k\times k}),\left\Vert\cdot\right\Vert_{\infty}\right) $ (see e.g. \cite{B2}, Theorem 2.4).
Hence there exists $\varepsilon > 0$ such that $\Vert B + Q - \tilde B\Vert \leq \varepsilon$ implies $\tilde B \in \mathrm{ED}^k$.
Let $n_0 \in \mathbb N$, such that $\Vert Q(n)\Vert < \varepsilon$ for $n > n_0$ and define
\begin{equation*}
    Q_\varepsilon(n)
    :=
    \begin{cases}
       Q(n), &n \leq n_0,
       \\
       0, &n > n_0.
    \end{cases}
\end{equation*}
Then $B + Q_\varepsilon \notin \mathrm{ED}^k$, since $Q_\varepsilon(n)$ is different from zero for only finitely many $n \in \mathbb N$, but $\Vert B + Q - (B + Q_\varepsilon)\Vert = \Vert Q - Q_\varepsilon\Vert < \varepsilon$ implies $B + Q_\varepsilon \in \mathrm{ED}^k$.
\end{proof}

\section{Main result}

In this section we prove the main result of this paper which states that the
interior of $\mathrm{BD}^{d} \subset \mathcal{L}^{\infty }(\mathbb{N},\mathbb{R}^{d\times d})$ equals $\mathrm{ED}^{d}$.

\begin{theorem}\label{INT_BD=ED}
$\operatorname{int} \mathrm{BD}^{d}=\mathrm{ED}^{d}$.
\end{theorem}

The following result is a reformulation of Theorem \ref{INT_BD=ED}, since $\mathrm{ED}^d \subseteq \mathrm{BD}^d$.

\begin{corollary}
The boundary $\partial \mathrm{BD}^{d}$ equals $(\operatorname{cl}\mathrm{BD}^{d})\backslash \mathrm{ED}^{d}.$
\end{corollary}

The set $\mathrm{BD}^{d}$ is not closed. Indeed,
consider the sequence $\left( A_{k}\right) _{k\in \mathbb{N}}$ in $\mathcal{L}^{\mathrm{Lya}}(\mathbb{N},\mathbb{R}^{d\times d})$ where $A_{k}$
is a constant sequence given by
\begin{equation*}
A_{k}=\operatorname{diag}\big[\mathrm e^{\frac{1}{k+1}},\dots,\mathrm e^{\frac{1}{k+1}}\big] .
\end{equation*}%
Each of the systems $A_{k}$ has a Bohl dichotomy (each of them has also
an exponential dichotomy), however, the limit system does not have a
Bohl dichotomy.

If \(d = 1\), then \(\mathrm{ED}^d = \mathrm{BD}^d\) (cf.\ \cite[Remark 24]{CzornikEtal2023}).
For \(d > 1\), however, we have the proper inclusion \(\mathrm{ED}^d \subsetneq \mathrm{BD}^d\).
For an example of a system \(A \in \mathrm{BD}^2 \setminus \mathrm{ED}^2\), confer\ \cite[Section 4]{Czornik2019}, and for \(d > 2\) consider for \(A \in \mathrm{BD}^2 \setminus \mathrm{ED}^2\) the system
\begin{equation*}
    \begin{pmatrix}
        A & 0 \\ 0 & \operatorname{diag}[2,\ldots,2]
    \end{pmatrix}
    \in \mathrm{BD}^d \setminus \mathrm{ED}^d.
\end{equation*}

We summarize the above considerations in the following Corollary.
\begin{corollary}
The set $\mathrm{BD}^{d}$ is not closed and $\operatorname{cl} \mathrm{ED}%
^{d}\nsubseteq \mathrm{BD}^{d}.$
If \(d > 1\), then $\mathrm{BD}^{d}$ is not open.
\end{corollary}

\begin{proof}
We have already mentioned that $\mathrm{BD}^{d}$ is not closed and to see $\operatorname{cl}\mathrm{ED}^{d} \nsubseteq \mathrm{BD}^{d}$, consider the sequence \((A_k)_{k\in\mathbb N}\) in \(\mathrm{ED}^d\) constructed above.
Lastly, let \(d > 1\) and assume to the contrary that \(\mathrm{BD}^d\) is open.
Then \(\mathrm{BD}^d = \operatorname{int}\mathrm{BD}^d = \mathrm{ED}^d\) by Theorem \ref{INT_BD=ED}.
This however, contradicts the proper inclusion \(\mathrm{ED}^d \subsetneq \mathrm{BD}^d\).
\end{proof}

\begin{proof}[Proof of Theorem \ref{INT_BD=ED}]
Let $A\in \mathrm{BD}^{d}\backslash \mathrm{ED}^{d}$ \ and let $L_{1}$ and $%
L_{2}$ be the subspaces from the definition of Bohl dichotomy. Then either $%
\underline{\beta }_{A}(L_{2})\leq 0$ or $\overline{\beta }_{A}(L_{1})$ $\geq
0,$ since otherwise $A\in \mathrm{ED}^{d}.$ We will show that in any neighborhood of $%
A$ there is a system which does not have a Bohl dichotomy. To do this let us
fix $\varepsilon >0.$ We will construct $Q\in \mathcal{L}^{\infty }(\mathbb{N%
},\mathbb{R}^{d\times d})$ such that $\left\Vert Q\right\Vert _{\infty }\leq
\varepsilon $ and the system $A+Q$ does not have a Bohl dichotomy.

If $\overline{\beta }_{A}(L_{1})$ $\geq 0,$ then denote by $k$ the
dimension of $L_{1}$ and consider the $L_{1}$-subsystem of (\ref{1}):
\begin{equation}
   y_{1}(n+1)=A_{L_1}(n)y_{1}(n).
   \label{GS14}
\end{equation}
By Lemma \ref{GS11} we know that
\begin{equation}
   \overline{\beta }_{A_{L_1}}(\mathbb{R}^{k})\geq 0.
   \label{GS12}
\end{equation}
By Lemma \ref{lem:BD-characterization}(ii), we know that
\begin{equation*}
   \sup_{x_{0}\in L_{1}\setminus \{0\}}\overline{\beta }_{A}(x_{0})<0.
\end{equation*}
The last inequality together with Lemma \ref{GS11} implies that
\begin{equation}
   \sup_{y_{0}\in \mathbb{R}^{k}\setminus \{0\}}\overline{\beta }_{A_{L_1}}(y_{0})<0.
   \label{GS13}
\end{equation}
Inequalities (\ref{GS12}) and (\ref{GS13}) show that system (\ref{GS14}) satisfies the assumptions of Lemma \ref{lem:lya_perturbed}.
When we fix a $y_{0,1}\in \mathbb{R}^{k}\setminus \{0\}$ and apply this Lemma to system \eqref{GS14} we obtain $Q_{1}\in \mathcal{L}^{\infty}(\mathbb{N},\mathbb{R}^{k\times k})$ with the following properties:
\begin{equation*}
   \lim_{\ell \rightarrow \infty }Q_{1}(\ell )=0,
   \qquad
   A_{L_1}+Q_{1}\in\mathcal{L}^{\mathrm{Lya}}(\mathbb{N},\mathbb{R}^{k\times k}),
\end{equation*}
and
\begin{equation}
   \underline{\beta }_{A_{L_1}+Q_{1}}(y_{0,1})<0
   \qquad\text{and}\qquad
   \overline{\beta }_{A_{L_1}+Q_{1}}(y_{0,1})\geq 0.
   \label{GS15}
\end{equation}
Now Lemma \ref{lem:subsystem_eps} gives us $Q\in \mathcal{L}^{\infty }(%
\mathbb{N},\mathbb{R}^{d\times d})$ and $x_{0}\in \mathbb{R}^{d}\backslash
\{0\}$ with $\left\Vert Q\right\Vert _{\infty }<\varepsilon $, $A+Q\in
\mathcal{L}^{\mathrm{Lya}}(\mathbb{N},\mathbb{R}^{k\times k})$ and
\begin{equation*}
\underline{\beta }_{A}(x_{0})\leq 0\leq \overline{\beta }_{A}(x_{0}).
\end{equation*}%
The last inequalities imply, in the light of Corollary \ref{cor:notBD}, that
$A+Q\notin \mathrm{BD}^{k}$.

Suppose now, that $\underline{\beta }_{A}(L_{2})\leq 0$. Denote by $k$
the dimension of $L_{2}$ and consider the $L_{2}$-subsystem of (\ref{1})%
\begin{equation}
   y_{2}(n+1)=A_{L_2}(n)y_{2}(n).
   \label{GS17}
\end{equation}
By Lemma \ref{GS11} we know that
\begin{equation}
   \underline{\beta }_{A_{L_2}}(\mathbb{R}^{k})\leq 0.
   \label{GS18}
\end{equation}
By Lemma \ref{lem:BD-characterization}(ii), we know that
\begin{equation*}
   \underset{x_{0}\in L_{2}\setminus \{0\}}{\inf }\underline{\beta }_{A}(x_{0}) > 0.
\end{equation*}
The last inequality together with Lemma \ref{GS11} implies that
\begin{equation}
   \underset{y_{0}\in \mathbb{R}^{k}\setminus \{0\}}{\inf }\underline{\beta}_{A_{L_2}}(y_{0})>0.
   \label{GS19}
\end{equation}
Inequalities \eqref{GS18} and \eqref{GS19} show that system \eqref{GS17} satisfies the assumptions of Lemma \ref{lem:lya_perturbed_lower_v2_step1}.
Applying this lemma we get $Q_{4}\in \mathcal{L}^{\infty }(\mathbb{N},\mathbb{R}^{k\times k})$ such that
\begin{equation*}
   \left\Vert Q_{4}\right\Vert _{\infty }<\frac{\varepsilon }{3},
   \qquad
   A_{L_2}+Q_{4}\in \mathcal{L}^{\mathrm{Lya}}(\mathbb{N},\mathbb{R}^{k\times k}),
\end{equation*}
and
\begin{equation}
   \inf_{x_{0}\in \mathbb{R}^{k}\setminus \{0\}}\underline{\beta }_{A_{L_2}+Q_{4}}(x_{0})>0
    \qquad\text{and}\qquad
   \underline{\beta }_{A_{L_2}+Q_{4}}(\mathbb{R}^{k})<0.
   \label{GS20}
\end{equation}
Observe that by Lemmas \ref{lem:BD-characterization} and \ref{lem:ED-characterization} the inequalities (\ref{GS20}) imply that $A_{L_2}+Q_{4}\in \mathrm{BD}^{k}\backslash \mathrm{ED}^{k}$.
The sequence $A_{L_2}+Q_{4}$ satisfies also the assumptions of Lemma \ref{Lem:10}, therefore there exists a $Q_{5}=\left( Q_{5}\left( n\right) \right)_{n\in \mathbb{N}}\in \mathcal{L}^{\infty }(\mathbb{N},\mathbb{R}^{k\times k})$ and $y_{0,2}\in \mathbb{R}^{k}\setminus \{0\}$ such that
\begin{equation*}
   \underset{l\rightarrow \infty }{\lim }Q_{5}(l)=0,
   \quad
   A_{L_2}+Q_4+Q_{5}\in \mathcal{L}^{\mathrm{Lya}}(\mathbb{N},\mathbb{R}^{k\times k})
\end{equation*}
and
\begin{equation}
   \underline{\beta }_{A_{L_2}+Q_{4}+Q_{5}}(y_{0,2}) < 0.
   \label{GS21}
\end{equation}
Without loss of generality we may assume that
\begin{equation*}
\left\Vert Q_{5}(n)\right\Vert <\frac{\varepsilon }{3},\qquad n\in \mathbb{N}.
\end{equation*}

Observe that $A_{L_2}+Q_{4}+Q_{5}\notin \mathrm{ED}^{k}$ by Lemma \ref{lem:ED_and_lim0}.
If
\begin{equation*}
\overline{\beta }_{A_{L_2}+Q_{4}+Q_{5}}(y_{0,2})\geq 0,
\end{equation*}
then the assumptions of Lemma \ref{lem:subsystem_eps} are satisfied and
therefore there exist $Q\in \mathcal{L}^{\infty }(\mathbb{N},\mathbb{R}%
^{d\times d})$ and $x_{0}\in \mathbb{R}^{d}\backslash \{0\}$ with $%
\left\Vert Q\right\Vert _{\infty }<\varepsilon $, $A+Q\in \mathcal{L}^{%
\mathrm{Lya}}(\mathbb{N},\mathbb{R}^{d\times d})$ and
\begin{equation*}
\underline{\beta }_{A}(x_{0})\leq 0\leq \overline{\beta }_{A}(x_{0})
\end{equation*}%
and in this case it follows from Corollary \ref{cor:notBD} that $A+Q\notin
\mathrm{BD}^{d}.$

So suppose that
\begin{equation*}
   \overline{\beta }_{A_{L_2}+Q_{4}+Q_{5}}(y_{0,2})<0.
\end{equation*}
If $A_{L_2}+Q_{4}+Q_{5}\notin \mathrm{BD}^{k}$, then we can obtain the desired perturbation by applying Lemma \ref{lem:subsystem_eps}.
Now consider the case $A_{L_2}+Q_{4}+Q_{5} \in \mathrm{BD}^{k} \setminus \mathrm{ED}^k$.
Then there are two subcases

1. For all $y_{0}\in \mathbb{R}^{k}\setminus \{0\}$ we have $\overline{\beta}_{A_{L_2}+Q_{4}+Q_{5}}(y_{0})<0.$

2. There exists $\overline{y}_{0}\in \mathbb{R}^{k}\setminus \{0\}$ such
that
\begin{equation}
   \overline{\beta }_{A_{L_2}+Q_{4}+Q_{5}}(\overline{y}_{0})>0.  \label{GS24}
\end{equation}

Note that we do not have to consider the case when there exists $\overline{y}_{0}\in
\mathbb{R}^{k}\setminus \{0\}$ such that $\overline{\beta }%
_{A_{L_2}+Q_{4}+Q_{5}}(\overline{y}_{0})=0,$ because then by Corollary \ref%
{cor:notBD} $A_{L_2}+Q_{4}+Q_{5}\notin \mathrm{BD}^{k}$ and we may again apply
Lemma \ref{lem:subsystem_eps} to get the desired $Q$.

In the first subcase
\begin{equation}
   \underset{y_{0}\in \mathbb{R}^{k}\setminus \{0\}}{\sup }\overline{\beta }_{A_{L_2}+Q_{4}+Q_{5}}(y_{0})\leq 0.  \label{GS25}
\end{equation}
The relation $A_{L_2}+Q_{4}+Q_{5}\notin \mathrm{ED}^{k}$ implies that
\begin{equation}
   \overline{\beta }_{A_{L_2}+Q_{4}+Q_{5}}\left( \mathbb{R}^{k}\right) \geq 0.
   \label{GS26}
\end{equation}

Inequalities (\ref{GS25}) and (\ref{GS26}) mean that $A_{L_2}+Q_{4}+Q_{5}$ satisfies the assumptions of Lemma \ref{lem:lya_perturbed_v2}.
Applying this lemma with $z_{0}=y_{0,2}$ we get  $Q_{6}\in \mathcal{L}^{\infty }(\mathbb{N},\mathbb{R}^{k\times k})$ with $\lim_{\ell \rightarrow \infty }Q_{6}(\ell )=0$, $A_{L_2}+Q_{4}+Q_{5}+Q_{6}\in \mathcal{L}^{\mathrm{Lya}}(\mathbb{N},\mathbb{R}^{k\times k})$ and
\begin{equation*}
   \underline{\beta }_{A_{L_2}+Q_{4}+Q_{5}+Q_{6}}(y_{0,2})\leq 0
   \qquad\text{and}\qquad
   \overline{\beta }_{A_{L_2}+Q_{4}+Q_{5}+Q_{6}}(y_{0,2})\geq 0.
\end{equation*}
Without loss of generality we may assume that
\begin{equation*}
   \left\Vert Q_{6}\right\Vert _{\infty }\leq \frac{\varepsilon }{3}.
\end{equation*}
System $A_{L_2}+Q_{4}+Q_{5}+Q_{6}$ satisfies the assumptions of Lemma \ref{lem:subsystem_eps} and therefore there exist $Q\in \mathcal{L}^{\infty }(\mathbb{N},\mathbb{R}^{d\times d})$ and $x_{0}\in \mathbb{R}^{d}\backslash
\{0\}$ with $\left\Vert Q\right\Vert _{\infty }<\varepsilon $, $A+Q \in \mathcal{L}^{\mathrm{Lya}}(\mathbb{N},\mathbb{R}^{d\times d})$ and
\begin{equation*}
   \underline{\beta }_{A}(x_{0})\leq 0\leq \overline{\beta }_{A}(x_{0}).
\end{equation*}
In this case it follows from Corollary \ref{cor:notBD} that $A+Q\notin\mathrm{BD}^{d}$.

Consider now the second case, i.e.\ the case when there exists $\overline{y}_{0}\in \mathbb{R}^{k}\setminus \{0\}$ such that (\ref{GS24}) is satisfied.
Note that $\underline{\beta}_{A_{L_2}+Q_4+Q_5}(\overline y_0) > 0$, as otherwise we would have a contradiction to $A_{L_2}+Q_4+Q_5 \in \mathrm{BD}^k$.
Hence $A_{L_2}+Q_{4}+Q_{5}$ satisfies the assumptions of Lemma \ref{lem:lya_perturbed_lower_v2_step3}.
Therefore there exists a $Q_{7} \in \mathcal{L}^{\infty }(\mathbb{N},\mathbb{R}^{k\times k})$ with $\left\Vert Q_{7}\right\Vert _{\infty}<\frac{\varepsilon }{3}$, $A_{L_2}+Q_{4}+Q_{5}+Q_{7}\in \mathcal{L}^{\mathrm{Lya}}(\mathbb{N},\mathbb{R}^{k\times k})$ and $A_{L_2}+Q_{4}+Q_{5}+Q_{7}\notin \mathrm{BD}^{k}$.
The desired perturbation may be now found by Lemma \ref{lem:subsystem_eps}.

In this way we have shown that in any neighborhood of a system $A\in \mathrm{BD}^{d}\backslash \mathrm{ED}^{d}$ there is a system that does not belong to $\mathrm{BD}^{d}$.
It implies that each system in $\mathrm{BD}^{d}\backslash\mathrm{ED}^{d}$ does not belong to $\operatorname{int}\mathrm{BD}^{d}$.
On the other hand the set $\mathrm{ED}^{d}$ is open in $\left( \mathcal{L}^{\mathrm{Lya}}(\mathbb{N},\mathbb{R}^{d\times d}),\left\Vert\cdot\right\Vert_{\infty}\right) $ (see e.g. \cite{B2}, Theorem 2.4) and therefore $\operatorname{int}\mathrm{BD}^{d}=\mathrm{ED}^{d}$.
\end{proof}

In \cite{CzornikEtal2023} the notion of Bohl Dichotomy Spectrum
\begin{equation*}
   \Sigma _{\mathrm{BD}}(A)
   \coloneqq
   \left\{ \gamma \in \mathbb{R}:x(n+1)=e^{-\gamma }A(n)x(n)%
    \text{ has no Bohl dichotomy}\right\}
\end{equation*}
and the corresponding resolvent $\varrho _{\mathrm{ED}}(A) \coloneqq \mathbb R\setminus \Sigma _{\mathrm{ED}}(A)$ is introduced.
Its relation to the Exponential Dichotomy Spectrum
\begin{equation*}
   \Sigma _{\mathrm{ED}}(A)
   \coloneqq
   \left\{ \gamma \in \mathbb{R} : x(n+1)=e^{-\gamma }A(n)x(n)%
    \text{ has no exponential dichotomy} \right\} 
\end{equation*}
and the corresponding resolvent $\varrho _{\mathrm{BD}}(A) \coloneqq \mathbb R\setminus \Sigma _{\mathrm{BD}}(A)$, is studied.
The following Corollary is an approximation result of the exponential dichotomy spectrum by the Bohl dichtomy spectrum.

\begin{corollary}[Approximating exponential by Bohl dichotomy spectra]
We have%
\begin{equation}
\underset{\varepsilon >0}{\bigcap }\text{ }\underset{%
\begin{array}{c}
Q\in \mathcal{L}^{\mathrm{\infty }}(\mathbb{N},\mathbb{R}^{d\times d}) \\
\left\Vert Q\right\Vert _{\infty }<\varepsilon
\end{array}%
}{\bigcup }\Sigma _{\mathrm{BD}}(A+Q)=\Sigma _{\mathrm{ED}}(A).  \label{c_1}
\end{equation}
\end{corollary}

\begin{proof}
Using definitions of Bohl dichotomy and exponential dichotomy resolvents as
well as the De Morgan's laws, equality (\ref{c_1}) may be rewritten as
follows%
\begin{equation*}
\underset{\varepsilon >0}{\bigcup }\text{ }\underset{%
\begin{array}{c}
Q\in \mathcal{L}^{\mathrm{\infty }}(\mathbb{N},\mathbb{R}^{d\times d}) \\
\left\Vert Q\right\Vert _{\infty }<\varepsilon
\end{array}%
}{\bigcap }\varrho _{\mathrm{BD}}(A+Q)=\varrho _{\mathrm{ED}}(A).
\end{equation*}%
Suppose that
\begin{equation*}
\gamma \in \underset{\varepsilon >0}{\bigcup }\text{ }\underset{%
\begin{array}{c}
Q\in \mathcal{L}^{\mathrm{\infty }}(\mathbb{N},\mathbb{R}^{d\times d}) \\
\left\Vert Q\right\Vert _{\infty }<\varepsilon
\end{array}%
}{\bigcap }\varrho _{\mathrm{BD}}(A+Q),
\end{equation*}%
then there exists $\varepsilon _{0}>0$ such that for all $Q\in \mathcal{L}^{%
\mathrm{\infty }}(\mathbb{N},\mathbb{R}^{d\times d})$ with $\left\Vert
Q\right\Vert _{\infty }<\varepsilon _{0}$ we have%
\begin{equation*}
\left( A+Q\right) \mathrm e^{-\gamma }\in \mathrm{BD}^{d}.
\end{equation*}%
The last relations implies that $A\mathrm e^{-\gamma }+\overline{Q}\in \mathrm{BD}%
^{d}$ for all $\overline{Q}\in \mathcal{L}^{\mathrm{\infty }}(\mathbb{N},%
\mathbb{R}^{d\times d})$ with $\left\Vert \overline{Q}\right\Vert _{\infty
}<\varepsilon _{0}\mathrm e^{\gamma }$ and therefore $A\mathrm e^{-\gamma }\in \operatorname{int}\mathrm{BD}%
^{d}.$ However, $\operatorname{int}\mathrm{BD}^{d}=\mathrm{ED}^{d}$ so $A\mathrm e^{-\gamma }\in
\mathrm{ED}^{d}$ and finally $\gamma \in \varrho _{\mathrm{ED}}(A).$ This
shows that
\begin{equation}
\underset{\varepsilon >0}{\bigcup }\text{ }\underset{%
\begin{array}{c}
Q\in \mathcal{L}^{\mathrm{\infty }}(\mathbb{N},\mathbb{R}^{d\times d}) \\
\left\Vert Q\right\Vert _{\infty }<\varepsilon
\end{array}%
}{\bigcap }\varrho _{\mathrm{BD}}(A+Q)\subset \varrho _{\mathrm{ED}}(A).
\label{c_3}
\end{equation}%
Suppose now, that $\gamma \in \varrho _{\mathrm{ED}}(A)$. It implies that $%
A\mathrm e^{-\gamma }\in \mathrm{ED}^{d}$ and since the set $\mathrm{ED}^{d}$ is
open in $\mathcal{L}^{\mathrm{Lya}}(\mathbb{N},\mathbb{R}^{d\times d})$ with
the metric induced by $\left\Vert \cdot\right\Vert _{\infty },$ then there is $%
\varepsilon _{1}>0$ such that
\begin{equation*}
A\mathrm e^{-\gamma }+Q=\left( A+Q\mathrm e^{\gamma }\right) \mathrm e^{-\gamma }\in \mathrm{ED}^{d}
\end{equation*}%
for all $Q\in \mathcal{L}^{\mathrm{\infty }}(\mathbb{N},\mathbb{R}^{d\times
d})$ with $\left\Vert Q\right\Vert _{\infty }<\varepsilon _{1}$. Since $%
\mathrm{ED}^{d}\subset \mathrm{BD}^{d},$ then
\begin{equation*}
\left( A+\widetilde{Q}\right)\mathrm e^{-\gamma }\in \mathrm{ED}^{d}
\end{equation*}%
for all $\widetilde{Q}\in \mathcal{L}^{\mathrm{\infty }}(\mathbb{N},\mathbb{R%
}^{d\times d})$ with $\left\Vert \widetilde{Q}\right\Vert _{\infty
}<\varepsilon _{1}\mathrm e^{\gamma }$ and consequently
\begin{equation*}
\gamma \in \underset{%
\begin{array}{c}
Q\in \mathcal{L}^{\mathrm{\infty }}(\mathbb{N},\mathbb{R}^{d\times d}) \\
\left\Vert Q\right\Vert _{\infty }<\varepsilon _{1}\mathrm e^{\gamma }%
\end{array}%
}{\bigcap }\varrho _{\mathrm{BD}}(A+Q).
\end{equation*}%
The last relation implies that
\begin{equation*}
\gamma \in \underset{\varepsilon >0}{\bigcup }\text{ }\underset{%
\begin{array}{c}
Q\in \mathcal{L}^{\mathrm{\infty }}(\mathbb{N},\mathbb{R}^{d\times d}) \\
\left\Vert Q\right\Vert _{\infty }<\varepsilon
\end{array}%
}{\bigcap }\varrho _{\mathrm{BD}}(A+Q)
\end{equation*}%
and finally, that the inclusion opposite to (\ref{c_3}) holds.
\end{proof}

\section{Appendix}

In this Appendix we describe the Millionshikov Rotation Method in the context of nonautonomous difference equations as a universal tool (see also \cite[Section 2]{Babiarz2017}). The method was developed by Millionshikov in the continuous-time case in  \cite{Millionscikov1967} (see also \cite{Izobov2012}).

For $x\in \mathbb{R}^{d}\backslash\{0\}$ and $\varepsilon \in \left[
0,\pi \right] $ denote the cone in direction $x$ with angle $\varepsilon$ by
\begin{equation*}
\operatorname{Con}\left[ x;\varepsilon \right] \coloneqq \left\{ y\in \mathbb{R}^{d}\backslash
\{0\}:\angle \left( x,y\right) \leq \varepsilon \right\} \cup \left\{
0\right\} ,
\end{equation*}%
where
\begin{equation*}
\angle \left( x,y\right) =\arccos \frac{\left\langle x,y\right\rangle }{%
\left\Vert x\right\Vert \left\Vert y\right\Vert }
\end{equation*}%
for $x,$ $y\in \mathbb{R}^{d}\backslash \{0\}$

\begin{definition}[$\varepsilon$-slow and $\varepsilon$-fast vectors of linear maps]
Let $F\bigskip :\mathbb{R}^{d}\rightarrow \mathbb{R}^{d}$ be linear.
An $x\in \mathbb{R}^{d}$ is called $\varepsilon $-slow for $F$
if
\begin{equation*}
\left\Vert Fx\right\Vert <\frac{\sin \varepsilon }{2}\left\Vert F\right\Vert
\left\Vert x\right\Vert .
\end{equation*}%
If
\begin{equation*}
\left\Vert Fx\right\Vert \geq \frac{\sin \varepsilon }{2}\left\Vert
F\right\Vert \left\Vert x\right\Vert ,
\end{equation*}%
then $x$ is called $\varepsilon$-fast for $F$. A $z\in \mathbb{R%
}^{d}$ is called maximal for $F$ if
\begin{equation*}
\left\Vert Fz\right\Vert =\left\Vert F\right\Vert \left\Vert z\right\Vert .
\end{equation*}
\end{definition}

The mapping $z\mapsto \Vert Fz\Vert $ on $\{z\in
\mathbb{R}^{d}:\Vert z\Vert =1\}$ is continuous and defined on a compact
set. Hence there is $z\in \mathbb{R}^{d}$, with $\Vert z\Vert =1$, such that
$\Vert Fz\Vert =\Vert F\Vert $ i.e.\ a maximal vector always exists.

\begin{lemma}[$\varepsilon$-fast vector in cone of $\varepsilon$-slow vector]
\label{MRML1}If $x\in \mathbb{R}^{d}$ is $\varepsilon -$slow for $F$, then
there exists $\overline{x}\in \operatorname{Con}\left[ x;\varepsilon \right] $ which is $%
\varepsilon -$fast for $F$.
\end{lemma}

\begin{proof}
Suppose that $x\in \mathbb{R}^{d}$ is $\varepsilon$-slow for $F$ and
consider $z\in \mathbb{R}^{d}$ which is a maximal vector for $F$. Such a
vector always exists. Note that $-z$ is also a maximal vector
for $F$. Consider the plane $\Pi = \operatorname{\operatorname{span}}\left\{ x,z\right\} $. In
the plane $\Pi $, the vector $x$ forms an angle not greater than $\frac{\pi }{2}$
with the vector $z$ or the vector $-z$. Without loss of generality, we can
assume that
\begin{equation*}
\gamma :=\angle \left( x,z\right) \leq \frac{\pi }{2}.
\end{equation*}%
If $\gamma \leq \varepsilon ,$ then $z\in \operatorname{Con}\left[ x;\varepsilon \right] .$
Let us therefore consider the case $\gamma >\varepsilon $. Without loss of
generality, we can assume that $\left\Vert x\right\Vert =\left\Vert
z\right\Vert =1.$ Let $\overline{x}\in \Pi$, $\left\Vert
\overline{x}\right\Vert =1$, be a vector
between $x$ and $z$ and forming with the vector $z$ the angle $\gamma
-\varepsilon $. The vector $\overline{x}$ can be represented in the form of a sum
\begin{equation}
\overline{x}=\alpha x+\beta z.  \label{MRM1}
\end{equation}%
Using $\langle x,z\rangle = \cos\gamma$, $\langle\overline x,z\rangle = \cos(\gamma-\varepsilon)$ and $\langle\overline x,x\rangle = \cos\varepsilon$ as well as the addition theorems for $\sin$ and $\cos$
\begin{equation*}
\beta = \frac{\sin\varepsilon}{\sin\gamma}
\qquad\text{and}\qquad
\alpha = \frac{\sin(\gamma-\varepsilon)}{\sin\gamma}.
\end{equation*}
Hence $\beta \geq \sin \varepsilon$ and thus $\frac{\alpha }{\beta }\leq \frac{1}{\sin \varepsilon }$.
Applying the operator $F$ to both sides of (\ref{MRM1}) we get
\begin{equation*}
F\overline{x}=\alpha Fx+\beta Fz
\end{equation*}%
and therefore
\begin{align*}
\left\Vert F\overline{x}\right\Vert
&\geq \beta \left\Vert Fz\right\Vert-\alpha \left\Vert Fx\right\Vert
\\
&= \beta \left\Vert Fz\right\Vert \left(1-\frac{\alpha \left\Vert Fx\right\Vert }{\beta\left\Vert Fz\right\Vert}\right)
\\
&\geq \sin\varepsilon\Vert F\Vert\bigg(1 - \frac 1{\sin\varepsilon}\frac{\frac{\sin\varepsilon}{2}\Vert F\Vert \Vert x\Vert}{\Vert F\Vert}\bigg)
= \frac{\sin\varepsilon}{2}\Vert F\Vert \Vert\overline x\Vert.
\end{align*}
Thus $\overline{x}\in \operatorname{Con}\left[ x;\varepsilon \right] $ is $%
\varepsilon$-fast for $F$.
\end{proof}

\begin{lemma}[Rewriting rotation with perturbation of transformation I]
\label{MRML2}
Suppose that $F \colon \mathbb R^d \to \mathbb R^d$ is linear and bijective.
Let $x,y \in \mathbb R^d \backslash \{0\}$ and $\varepsilon \in \left( 0,\pi \right]$
are such that $y\in \operatorname{Con}\left[ x;\varepsilon \right]$ and $\left\Vert x\right\Vert =\left\Vert y\right\Vert$.
Then there exists a linear mapping $Q:\mathbb{R}^{d}\rightarrow \mathbb{R}^{d}$ such that $\left\Vert
Q\right\Vert \leq \varepsilon \left\Vert F\right\Vert ,$ $\left( F+Q\right)F^{-1}x = y$ and $F+Q$ is bijective.
\end{lemma}

\begin{proof}
We assume that $\Vert x\Vert = \Vert y\Vert = 1$.
Let $V:\mathbb{R}^{d}\rightarrow \mathbb{R}^{d}$ be in the special
orthogonal group, with $Vx=y$ and which is the identity on $\operatorname{\operatorname{span}}\left\{
x,y\right\} ^{\perp }$. We define $Q:=(V-I)F$. For $z\in \mathbb{R}^{d}$
it holds that
\begin{equation*}
(F+Q)z=\big(F+(V-I)F\big)z=VFz,
\end{equation*}%
from which it follows that $F+Q$ is bijective and $(F+Q)F^{-1}x = y$. Since
\begin{equation*}
\Vert Q\Vert =\Vert (V-I)F\Vert \leq \Vert V-I\Vert \cdot \Vert
F\Vert .
\end{equation*}%
We have to show that $\Vert V-I\Vert \leq \varepsilon $.
By the cosine series, it holds that $\cos \varepsilon \geq 1-\frac{\varepsilon ^{2}}{2}$.
Moreover, $V-I$ restricted to $\operatorname{\operatorname{span}}\{x,y\}^\perp$ is zero, and on the plane $\operatorname{\operatorname{span}}\{x,y\}$, $V$ acts as rotation.
Hence, $\Vert V-I\Vert = \Vert (V-I)x\Vert$ and
\begin{equation*}
\Vert (V-I)x\Vert^{2}
= \Vert y-x\Vert
^{2}=2-2\langle y,x\rangle \leq 2(1-\cos \varepsilon )\leq \varepsilon^{2}.
\end{equation*}
\end{proof}

\begin{lemma}[Rewriting rotation with perturbation of transformation II]
\label{MRML3}
Suppose that $F \colon \mathbb R^d \to \mathbb R^d$ is linear and bijective.
Let $x,y \in \mathbb R^d \backslash \{0\}$ and $\varepsilon \in \left( 0,\pi \right]$
are such that $y\in \operatorname{Con}\left[ x;\varepsilon \right]$ and $\left\Vert x\right\Vert =\left\Vert y\right\Vert$.
Then there exists a linear mapping $Q:\mathbb{R}^{d}\rightarrow \mathbb{R}^{d}$ such that $\left\Vert
Q\right\Vert \leq \varepsilon \left\Vert F\right\Vert ,$ $\left( F+Q\right)y = Fx$ and $F+Q$ is bijective.
\end{lemma}

\begin{proof}
To obtain the proof we can repeat the arguments from the proof of Lemma \ref%
{MRML2} with $Q:=F(V-I).$
\end{proof}

Consider now a system%
\begin{equation}
x(n+1)=A(n)x(n)  \label{MRM4}
\end{equation}%
with $(A(n))_{n\in \mathbb{N}}\in \mathcal{L}^{\mathrm{Lya}}(\mathbb{N},%
\mathbb{R}^{d\times d}).$

\begin{lemma}[From $\varepsilon$-slow to $\varepsilon$-fast]
\label{MRML4}Suppose that $\varepsilon \in \left( 0,\frac{\pi }{2}\right) $,
$k,m\in \mathbb{N}$, $k<m$ and $x_{0}\in \mathbb{R}^{d}\backslash \left\{
0\right\} .$ Then we have

\begin{itemize}
\item If a solution $\left( x\left( n,x_{0}\right) \right) _{n\in \mathbb{N%
}}$ of system (\ref{MRM4}) satisfies%
\begin{equation}
\Vert x(m,x_{0})\Vert <\frac{\sin \varepsilon }{2}\Vert \Phi _{A}(m,k)\Vert
\,\left\Vert x(k,x_{0})\right\Vert ,  \label{MRM5}
\end{equation}%
then there exists $x_{k}\in \operatorname{Con}\left[ x\left( k,x_{0}\right) ;\varepsilon %
\right] $ such that the solution $\left( \overline{x}\left(
n,k,x_{k}\right) \right) _{n\in \mathbb{N}}$ of system (\ref{MRM4})
satisfies
\begin{equation}
\Vert \overline{x}(m,k,x_{k})\Vert \geq \frac{\sin \varepsilon }{2}\Vert
\Phi _{A}(m,k)\Vert \,\left\Vert x_{k}\right\Vert .  \label{MRM6}
\end{equation}

\item If a solution $\left( x\left( n,x_{0}\right) \right) _{n\in \mathbb{N%
}}$ of system (\ref{MRM4}) satisfies%
\begin{equation}
\Vert x(k,x_{0})\Vert <\frac{\sin \varepsilon }{2}\Vert \Phi _{A}(k,m)\Vert
\,\left\Vert x(m,x_{0})\right\Vert ,  \label{MRM7}
\end{equation}%
then here exists $x_{m}\in \operatorname{Con}\left[ x\left( m,x_{0}\right) ;\varepsilon %
\right] $ such that the solution $\left( \overline{%
\overline{x}}\left( n,m,x_{m}\right) \right) _{n\in \mathbb{N}}$ of system (%
\ref{MRM4}) satisfies
\begin{equation}
\Vert \overline{\overline{x}}(k,m,x_{m})\Vert \geq \frac{\sin \varepsilon }{2%
}\Vert \Phi _{A}(k,m)\Vert \,\left\Vert x_{m}\right\Vert .  \label{MRM8}
\end{equation}
\end{itemize}
\end{lemma}

\begin{proof}
Let us denote by $F$ the mapping induced by the matrix $\Phi _{A}\left( m,k\right)$ and $x = x(k,x_{0})$.
Then inequality \eqref{MRM5} states that $x$ is $\varepsilon$-slow for $F$ and
therefore by Lemma \ref{MRML1} there exists $\overline{x}\in \operatorname{Con}\left[
x;\varepsilon \right] $ which is $\varepsilon$-fast for $F$ and which yields \eqref{MRM6} for $x_{k}\coloneqq\overline{x}$.
Similarly, applying Lemma \ref{MRML1} to the mapping $F$ induced by the matrix $\Phi_{A}\left(k,m\right)$ and $x = x(m,x_{0})$, we obtain \eqref{MRM8}.
\end{proof}

For a sequence $(Q(n))_{n\in \mathbb{N}}$ in $\mathbb R^{d\times d}$, consider the so-called perturbed system of system \eqref{MRM4},
\begin{equation}
z(n+1)=\left( A(n)+Q(n)\right) z(n).  \label{MRM9}
\end{equation}

\begin{lemma}[Perturbation at fixed time in cone]
\label{MRML5}Suppose that $\varepsilon \in \left( 0,\frac{\pi }{2}\right) $,
$k,m\in \mathbb{N}$, $1\leq k<m$ and $x_{0}\in \mathbb{R}^{d}\backslash
\left\{ 0\right\} .$ Then we have

\begin{itemize}
\item If $\left( x\left( n,x_{0}\right) \right) _{n\in \mathbb{N}}$ is a
solution of system (\ref{MRM4}) and $x_{k}\in \operatorname{Con}\left[ x\left(
k,x_{0}\right) ;\varepsilon \right] ,$ $\left\Vert x_{k}\right\Vert
=\left\Vert x\left( k,x_{0}\right) \right\Vert $, then there exists a
sequence $(Q(n))_{n\in \mathbb{N}}$ in $\mathbb R^{d\times d}$ such that $%
Q(n)=0$ for $n\neq k-1,$ $\left\Vert Q(k-1)\right\Vert \leq \varepsilon
\left\Vert A(k-1)\right\Vert $ and such that for the solution $\left(
z\left( n\right) \right) _{n\in \mathbb{N}}$ of system (\ref{MRM9}) with $%
z(k-1)=x(k-1,x_{0}),$ we have $z(k)=x_{k}$ and $(A(n)+Q(n))_{n\in \mathbb{N}%
}\in \mathcal{L}^{\mathrm{Lya}}(\mathbb{N},\mathbb{R}^{d\times d})$.

\item If $\left( x\left( n,x_{0}\right) \right) _{n\in \mathbb{N}}$ is a
solution of system (\ref{MRM4}) and $x_{m}\in \operatorname{Con}\left[ x\left(
m,x_{0}\right) ;\varepsilon \right] $, $\left\Vert x_{m}\right\Vert
=\left\Vert x\left( m,x_{0}\right) \right\Vert $, then there exists a
sequence $(Q(n))_{n\in \mathbb{N}}$ in $\mathbb R^{d\times d}$ such that $Q(n)=0$
for $n\neq m,$ $\left\Vert Q(m)\right\Vert \leq \varepsilon \left\Vert
A(m)\right\Vert $ and such that for the solution $\left( z\left( n\right)
\right) _{n\in \mathbb{N}}$ of system (\ref{MRM9}) with $%
z(m+1)=x(m+1,x_{0}), $ we have $z(m)=x_{m}$ and $(A(n)+Q(n))_{n\in \mathbb{N}%
}\in \mathcal{L}^{\mathrm{Lya}}(\mathbb{N},\mathbb{R}^{d\times d})$.
\end{itemize}
\end{lemma}

\begin{proof}
According to Lemma \ref{MRML2} with $F$ being the mapping induced by the matrix $A(k-1)$, $x=x(k,x_{0})$ and $y=x_{k}$,
there exists $Q'$ in $\mathbb R^{d\times d}$ such that $\Vert Q'\Vert \leq \varepsilon \Vert A(k-1)\Vert$, such that
\begin{equation*}
    \big(A(k-1)+Q'\big)A(k-1)^{-1}x(k,x_0) = x_k,
    \quad\text{i.e.}\quad
    \big(A(k-1)+Q'\big)x(k-1,x_0) = x_k
\end{equation*}
and such that $A(k-1) + Q'$ is bijective.
Then the first point of the Lemma is satisfied for $(Q(n))_{n\in \mathbb{N}}$ in $\mathbb R^{d\times d}$ given by
\begin{equation*}
Q(n) =
\begin{cases}
Q' & \text{for $n = k-1$},
\\
0 & \text{otherwise}.
\end{cases}
\end{equation*}

To prove the second point we use Lemma \ref{MRML3} with $F$ being the mapping induced by the matrix $A(m)$, $x=x(m,x_{0})$ and $y=x_{m}$.
Hence there exists $Q'$ in $\mathbb{R}^{d\times d}$ such that $\Vert Q'\Vert \leq \varepsilon \Vert A(m)\Vert$, such that
\begin{equation*}
    \big(A(m)+Q'\big)x_m = A(m)x(m,x_{0}) = x(m+1,x_0)
\end{equation*}
and such $A(m) + Q'$ is bijective.
To conclude the proof, consider $(Q(n))_{n\in \mathbb{N}}$ in $\mathbb R^{d\times d}$ given by
\begin{equation*}
Q(n) =
\begin{cases}
Q' & \text{for $n = m$},
\\
0 & \text{otherwise}.
\end{cases}
\end{equation*}
\end{proof}

From Lemmas \ref{MRML4} and \ref{MRML5} we obtain Millionschikov's method of
rotations in a dynamic version (cp.\ also Remark \ref{rem:millionshikov_alg} for an algebraic formulation).

\begin{theorem}[Millionshikov Rotation Method]
\label{MRMC1}
Let $\varepsilon >0$, $k,$ $m\in \mathbb{N}$, $m>k$ and $%
x_{0}\in \mathbb{R}^{d}\backslash \left\{ 0\right\}$. Then

\begin{enumerate}
\item[(a)] (\textbf{Forward Millionshikov Rotation Method}) There exists a sequence $(Q(n))_{n\in \mathbb{N}}$ in $\mathbb R^{d\times d}$ such that $Q(n)=0$ for $n\neq k-1,$ $\left\Vert
Q(k-1)\right\Vert \leq \varepsilon \left\Vert A(k-1)\right\Vert ,$ $%
(A(n)+Q(n))_{n\in \mathbb{N}}\in \mathcal{L}^{\mathrm{Lya}}(\mathbb{N},%
\mathbb{R}^{d\times d})$ and such that the solution $(z(n))_{n\in \mathbb{N}%
} $ of the perturbed system (\ref{MRM9}) with $z(k-1)=x(k-1,x_{0})$ satisfies%
\begin{equation}
\Vert z(m)\Vert \geq \frac{\sin \varepsilon }{2}\Vert \Phi _{A}(m,k)\Vert
\,\left\Vert z(k)\right\Vert  \label{MRM10}
\end{equation}%
and%
\begin{equation}
\left\Vert x(k,x_{0})\right\Vert =\Vert z(k)\Vert .  \label{MRM11}
\end{equation}

\item[(b)] (\textbf{Backward Millionshikov Rotation Method}) There exists a sequence $(Q(n))_{n\in \mathbb{N}}$ in $\mathbb R^{d\times d}$ such that $Q(n)=0$ for $n\neq m,$ $\left\Vert Q(m)\right\Vert \leq
\varepsilon \left\Vert A(m)\right\Vert ,(A(n)+Q(n))_{n\in \mathbb{N}}\in
\mathcal{L}^{\mathrm{Lya}}(\mathbb{N},\mathbb{R}^{d\times d})$ and such that
the solution $(z(n))_{n\in \mathbb{N}}$ of the perturbed system (\ref{MRM9})
with $z(m+1)=x(m+1,x_{0})$ satisfies%
\begin{equation}
\Vert z(k)\Vert \geq \frac{\sin \varepsilon }{2}\Vert \Phi _{A}(k,m)\Vert
\,\left\Vert z(m)\right\Vert  \label{MRM12}
\end{equation}%
and%
\begin{equation}
\left\Vert x(m,x_{0})\right\Vert =\Vert z(m)\Vert .  \label{MRM13}
\end{equation}
\end{enumerate}
\end{theorem}

\begin{proof}
First we will prove the point 1. If for the solution $\left( x\left(
n,x_{0}\right) \right) _{n\in \mathbb{N}}$ of system (\ref{MRM4}) we have
\begin{equation*}
\Vert x(m,x_{0})\Vert \geq \frac{\sin \varepsilon }{2}\Vert \Phi
_{A}(m,k)\Vert \,\left\Vert x(k,x_{0})\right\Vert ,
\end{equation*}%
then $Q(n)=0,$ $n\in \mathbb{N}$ is the desired sequence. Suppose now that
\begin{equation*}
\Vert x(m,x_{0})\Vert <\frac{\sin \varepsilon }{2}\Vert \Phi _{A}(m,k)\Vert
\,\left\Vert x(k,x_{0})\right\Vert .
\end{equation*}%
According to point 1 of Lemma \ref{MRML4} there exists $x_{k}\in \operatorname{Con}\left[
x\left( k,x_{0}\right) ;\varepsilon \right] $ such that the solution $\left( \overline{x}\left( n,k,x_{k}\right) \right) _{n\in
\mathbb{N}}$ of system (\ref{MRM4}) satisfies (\ref{MRM6}). Since for any $%
\alpha \in \mathbb{R}\backslash \left\{ 0\right\} $ we have $\overline{x}%
(m,k,\alpha x_{k})=\alpha \overline{x}(m,k,x_{k})$, then we may assume that $%
\left\Vert x(k,x_{0})\right\Vert =\Vert x_{k}\Vert .$ We will show that the
sequence $(Q(n))_{n\in \mathbb{N}}$ from point 1 of Lemma \ref{MRML5} is the
desired one. From point 1 of Lemma \ref{MRML5} we know that $Q(n)=0$ for $%
n\neq k-1,$ $\left\Vert Q(k-1)\right\Vert \leq \varepsilon \left\Vert
A(k-1)\right\Vert ,$ $(A(n)+Q(n))_{n\in \mathbb{N}}\in \mathcal{L}^{\mathrm{%
Lya}}(\mathbb{N},\mathbb{R}^{d\times d}).$ Consider the solution $%
(z(n))_{n\in \mathbb{N}}$ of the perturbed system (\ref{MRM9}) with $%
z(k-1)=x(k-1,x_{0}).$ Then%
\begin{equation*}
\Vert z(m)\Vert =\Vert \Phi _{A}(m,k)z(k)\Vert =
\end{equation*}%
\begin{equation*}
\Vert \Phi _{A}(m,k)x_{k}\Vert =\left\Vert \overline{x}\left(
m,k,x_{k}\right) \right\Vert \overset{\text{by (\ref{MRM6}) }}{\geq }
\end{equation*}%
\begin{equation*}
\frac{\sin \varepsilon }{2}\Vert \Phi _{A}(m,k)\Vert \,\left\Vert
x_{k}\right\Vert =\frac{\sin \varepsilon }{2}\Vert \Phi _{A}(m,k)\Vert
\,\left\Vert z(k)\right\Vert ,
\end{equation*}%
since from point 1 of Lemma \ref{MRML5} we know that $z(k)=x_{k}.$ The last
inequality proves (\ref{MRM10}). Equality (\ref{MRM11}) follows from $%
z(k)=x_{k}$ and $\left\Vert x(k,x_{0})\right\Vert =\Vert x_{k}\Vert .$

The proof of point 2 is similar. If for the solution $\left( x\left(
n,x_{0}\right) \right) _{n\in \mathbb{N}}$ of system (\ref{MRM4}) we have
\begin{equation*}
\Vert x(k,x_{0})\Vert \geq \frac{\sin \varepsilon }{2}\Vert \Phi
_{A}(k,m)\Vert \,\left\Vert x(m,x_{0})\right\Vert ,
\end{equation*}%
then $Q(n)=0,$ $n\in \mathbb{N}$ is the desired sequence. Suppose now
that%
\begin{equation*}
\Vert x(k,x_{0})\Vert <\frac{\sin \varepsilon }{2}\Vert \Phi _{A}(k,m)\Vert
\,\left\Vert x(m,x_{0})\right\Vert .
\end{equation*}%
According to point 2 of Lemma (\ref{MRML4}) there exists $x_{m}\in \operatorname{Con}\left[
x\left( m,x_{0}\right) ;\varepsilon \right] $ such that the solution $\left( \overline{x}\left( n,m,x_{m}\right) \right) _{n\in
\mathbb{N}}$ of system (\ref{MRM4}) satisfies (\ref{MRM8}). Since for any $%
\alpha \in \mathbb{R}\backslash \left\{ 0\right\} $ we have $\overline{x}%
(k,m,\alpha x_{m})=\alpha \overline{x}(k,m,x_{m}),$ then we may assume that $%
\left\Vert x(m,x_{0})\right\Vert =\Vert x_{m}\Vert .$ We will show that the
sequence $(Q(n))_{n\in \mathbb{N}}$ from point 2 of Lemma \ref{MRML5} is the
desired one. From point 2 of Lemma \ref{MRML5} we know that $Q(n)=0$ for $%
n\neq m,$ $\left\Vert Q(m)\right\Vert \leq \varepsilon \left\Vert
A(m)\right\Vert ,$ $(A(n)+Q(n))_{n\in \mathbb{N}}\in \mathcal{L}^{\mathrm{Lya%
}}(\mathbb{N},\mathbb{R}^{d\times d}).$ Consider the solution $(z(n))_{n\in
\mathbb{N}}$ of the perturbed system (\ref{MRM9}) with $z(m+1)=x(m+1,x_{0}).$
From point 2 of Lemma \ref{MRML5} we know that $z(m)=x_{m}$ and therefore%
\begin{equation*}
\Vert z(k)\Vert =\Vert \Phi _{A}(k,m)z(m)\Vert =\Vert \Phi
_{A}(k,m)x_{m}\Vert =
\end{equation*}%
\begin{equation*}
\left\Vert \overline{x}\left( k,m,x_{m}\right) \right\Vert \overset{\text{by
(\ref{MRM8}) }}{\geq }
\end{equation*}%
\begin{equation*}
\frac{\sin \varepsilon }{2}\Vert \Phi _{A}(k,m)\Vert \,\left\Vert
x_{m}\right\Vert =\frac{\sin \varepsilon }{2}\Vert \Phi _{A}(k,m)\Vert
\,\left\Vert z(m)\right\Vert.
\end{equation*}
The last inequality proves (\ref{MRM12}). Equality (\ref{MRM13}) follows
from $z(m)=x_{m}$ and $\left\Vert x(m,x_{0})\right\Vert =\Vert x_{m}\Vert .$
\end{proof}

Theorem \ref{MRMC1} can be reformulated into an algebraic version.

\begin{remark}[Millionshikov Rotation Method, algebraic version]
\label{rem:millionshikov_alg}\hfill
\begin{itemize}
\item[(a)]
(\textbf{Forward Millionshikov Rotation Method})
Let $n,m\in \mathbb{N}$ with $m<n$, $B(m),\dots
,B(n)\in \mathrm{GL}(k)$, $v\in \mathbb{R}^{k}$ and $\varepsilon >0$. Then
there exists $R\in \mathbb{R}^{k\times k}$ with

(i) $\Vert R\Vert \leq \varepsilon \cdot \max \big\{\Vert B(m)\Vert ,\Vert
B(m)^{-1}\Vert \big\}$,

(ii) $B(m)+R\in \mathrm{GL}(k)$

(iii) $\Vert B(n)\cdots B(m+1)\left( B(m)+R\right) v\Vert \geq \frac{\sin
\varepsilon }{2}\Vert B(n)\cdots B(m+1)\Vert \cdot \Vert B(m)v\Vert $,

(iv) $\Vert B(m)v\Vert =\Vert \left( B(m)+R\right) v\Vert $.

\item[(b)]
(\textbf{Backward Millionshikov Rotation Method})
Let $n,m\in \mathbb{N}$ with $m<n$, $B(m),\dots
,B(n)\in \mathrm{GL}(k)$, $v\in \mathbb{R}^{k}$ and $\varepsilon >0$. Then
there exists $R\in \mathbb{R}^{k\times k}$ with

(i) $\Vert R\Vert \leq \varepsilon \cdot \max \big\{\Vert B(n)\Vert ,\Vert
B(n)^{-1}\Vert \big\}$,

(ii) $B(n)+R\in \mathrm{GL}(k)$

(iii) $\Vert B^{-1}(m)\cdots B^{-1}(n-1)\left( B(n)+R\right) ^{-1}v\Vert \geq
\frac{\sin \varepsilon }{2}\Vert B^{-1}(m)\cdots B^{-1}(n-1)\Vert \cdot
\Vert B^{-1}(n)v\Vert $,

(iv) $\Vert B^{-1}(n)v\Vert =\Vert \left( B(n)+R\right) ^{-1}v\Vert $.
\end{itemize}
\end{remark}

\end{document}